\documentclass[11pt,reqno]{amsart}

\usepackage{latexsym}
\usepackage{amssymb}
\usepackage[cp850]{inputenc}
\usepackage{epsfig}
\usepackage{psfrag}
\usepackage{amsthm}
\usepackage{amscd}
\usepackage{amsmath}
\usepackage{amsfonts}
\usepackage{graphics}
\usepackage{hyperref}

\usepackage{epstopdf}

\theoremstyle{plain}
\newtheorem{thm}{Theorem} 
\newtheorem*{thmA}{Theorem \ref{th:A}} 
\newtheorem*{thmB}{Theorem \ref{th:B}} 
\newtheorem*{thmC}{Theorem \ref{th:C}}
\newtheorem*{defiV}{Definition \ref{defi:volume}}
\newtheorem{lem}[thm]{Lemma}

\newtheorem{cor}[thm]{Corollary}
\newtheorem{prop}[thm]{Proposition} 
\newtheorem*{thm*}{Theorem}
\newtheorem*{cor*}{Corollary}

\newtheorem*{prop*}{Proposition}
\newtheorem*{claim*}{Claim}

\theoremstyle{definition}
\newtheorem{ex}[thm]{Example}
\newtheorem{defi}[thm]{Definition} 
\newtheorem*{defi*}{Definition}
\newtheorem*{ex*}{Example} 
\newtheorem*{rmk*}{Remark}
\newtheorem{rmk}[thm]{Remark}

\let\ssection=\section
\renewcommand{\section}{\setcounter{equation}{0}\ssection}

\newtheorem*{namedtheorem}{\theoremname}
\newcommand{\theoremname}{testing}

\theoremstyle{remark}

\numberwithin{thm}{section}

\newcommand{\BR}{\mathbb R}			
			
			\newcommand{\BZ}{\mathbb Z}

\newcommand{\CA}{\mathcal A}		\newcommand{\CB}{\mathcal B}
\newcommand{\CC}{\mathcal C}		\newcommand{\calD}{\mathcal D}
\newcommand{\CE}{\mathcal E}		
\newcommand{\CG}{\mathcal G}		\newcommand{\CH}{\mathcal H}

		\newcommand{\CT}{\mathcal T}
		\newcommand{\CV}{\mathcal V}
		\newcommand{\CX}{\mathcal X}
		\newcommand{\CZ}{\mathcal Z}

\newcommand{\co}{\colon\thinspace}

\newcommand{\la}{\langle}
\newcommand{\ra}{\rangle}
\newcommand{\ocap}{{\buildrel \circ \over\cap \, }}

\DeclareMathOperator{\Mod}{MCG}		
\DeclareMathOperator{\Out}{Out}		
\DeclareMathOperator{\Aut}{Aut}		
\DeclareMathOperator{\IA}{IA}			
\DeclareMathOperator{\SL}{SL}		
\DeclareMathOperator{\GL}{GL}		
\DeclareMathOperator{\Sp}{Sp}		
\DeclareMathOperator{\vol}{vol}		
\DeclareMathOperator{\rank}{rank}		

\begin{document}

\title[Twisting out fully irreducible automorphisms]{Twisting out
  fully irreducible automorphisms} 

\author[M.~Clay]{Matt Clay}
\address{Dept.\ of Mathematics \\
  Allegheny College\\
  Meadville, PA 16335}
\email{mclay@allegheny.edu}

\author[A.~Pettet]{Alexandra Pettet}
\address{Dept.\ of Mathematics \\
  University of Michigan\\
  Ann Arbor, MI 48109}
\email{apettet@umich.edu}

\date{\today}
\thanks{\tiny The second author is partially supported by NSF grant
  DMS-0856143 and NSF RTG DMS-0602191}

\begin{abstract}
  By a theorem of Thurston, in the subgroup of the mapping class group
  generated by Dehn twists around two curves which fill, every element
  not conjugate to a power of one of the twists is pseudo-Anosov. We
  prove an analogue of this theorem for the outer automorphism group
  of a free group.
  \end{abstract}

\maketitle

\section{Introduction}\label{sc:intro}

A {\it fully irreducible} element of the outer automorphism group
$\Out F_k$ of a free group $F_k$ is characterized by the property that
no nontrivial power fixes the conjugacy class of a proper free factor
of $F_k$. Considered to be analogous to pseudo-Anosov elements of the
mapping class group (see \cite{Casson-Bleiler} or \cite{FLP}), fully
irreducible elements play a similarly important role in the study of
$\Out F_k$. Levitt and Lustig \cite{Levitt-Lustig} showed for instance
that fully irreducible elements exhibit North-South dynamics on the
closure of Culler--Vogtmann's Outer space, the projectivized space of
minimal very small actions of $F_k$ on $\BR$--trees
\cite{Bestvina-Feighn,Cohen-Lustig}. More recently Algom-Kfir
\cite{Algom-Kfir} proved that axes of fully irreducibles in Outer
space, equipped with the Lipschitz metric, are strongly contracting,
indicating that this class of outer automorphisms should be useful
towards understanding the geometry of $\Out F_k$.

In this paper we present a method for constructing fully irreducible
elements of $\Out F_k$. Our approach is to replicate the following
result of Thurston concerning pseudo-Anosov mapping classes: a pair of
Dehn twists around filling simple closed curves generate a nonabelian
free group in which any element not conjugate to a power of one of the
twists is pseudo-Anosov \cite{ar:T88}. The irreducible outer
automorphisms we construct have the additional property of being {\it
  atoroidal}; that is, none of their nontrivial powers fix a conjugacy
class of $F_k$. By theorems of Bestvina--Feighn
\cite{Bestvina-Feighn_Combination}, Brinkmann \cite{Brinkmann}, and
Gersten \cite{Gersten_Iso}, the atoroidal elements of $\Out F_k$ are
precisely the {\it hyperbolic} elements, consisting of exactly those
elements with hyperbolic mapping tori, and so we will use only the
latter term.

Before stating precisely our main theorem, we briefly recall some
known constructions of fully irreducible elements of $\Out F_k$.

\medskip
\noindent {\bf Geometric:} By Thurston's theorem, we obtain
pseudo-Anosov homeomorphisms from two Dehn twists around filling
curves on a surface $S$ with a single boundary component. From an
identification $\pi_1 (S) \cong F_k$, any such pseudo-Anosov induces a
fully irreducible outer automorphism of $F_k$. Such an outer
automorphism is necessarily not hyperbolic as the conjugacy class of
the element of $F_k$ corresponding to the boundary component of $S$ is
periodic.  

\medskip
\noindent {\bf Homological:} As in the case of the mapping class group
\cite{Casson-Bleiler,Margalit-Spallone}, there is a homological
criterion that ensures an outer automorphism is fully irreducible.
Namely, Gersten and Stallings \cite{Gersten-Stallings} gave algebraic
criteria for fully irreducibility, providing sufficient conditions in
terms of the matrix corresponding to the action of the outer
automorphism on the homology of $F_k$.  This provides examples of
hyperbolic fully irreducible elements, but the action on homology is
necessarily nontrivial.

\medskip

Our construction begins with an analogy to surfaces: a simple closed
curve on a surface determines a splitting of the surface group over
the cyclic subgroup generated by the curve. For $\Out F_k$, the role
of a simple closed curve can be taken by a splitting of $F_k$ over a
cyclic subgroup. We prove using an appropriate notion of Dehn twist
automorphism (defined by a splitting of $F_k$ over a cyclic subgroup)
and of filling splittings:

\begin{thmA}
  Let $\delta_1$ and $\delta_2$ be the Dehn twists of
  $F_k$ for two filling cyclic splittings of $F_k$. Then
  there exists $N = N(\delta_1,\delta_2)$ such that for $m,n >N$:
  \begin{enumerate} 
  \item $\la \delta_1^m, \delta_2^n \ra$ is isomorphic to the
    free group on two generators; and
  
  \item if $\phi \in \la \delta_1^m, \delta_2^n \ra$ is not conjugate
    to a power of either $\delta^m_1$ or $\delta^n_2$, then $\phi$ is
    a hyperbolic fully irreducible element of $\Out F_k$.
  \end{enumerate}
\end{thmA}

Kapovich-Lustig \cite{KL_Pingpong} and Hamenst\"adt \cite{Hamenstadt} have recently, using different methods, given another construction of hyperbolic fully irreducible elements in $\Out F_k$.
Namely, they show that given two hyperbolic fully irreducible elements
$\phi, \psi \in \Out F_k$, then either the subgroup $\la \phi, \psi
\ra$ is virtually cyclic, or there is a constant $N = N(\phi,\psi)$
such that for all $m,n \geq N$ the subgroup $\la \phi^m, \psi^n \ra$ is
isomorphic to the free group on two generators and every nontrivial
element of the subgroup is hyperbolic and fully irreducible.

Theorem \ref{th:A} produces new examples of fully irreducible
elements, not attained by previous methods. For instance, we can
construct examples of hyperbolic (and therefore obtainable by
Thurston's theorem) fully irreducible elements that act trivially on
homology. Papadopoulos used Thurston's construction of pseudo-Anosov
homeomorphisms that act trivially on homology to construct for any
symplectic matrix in $\Sp(2g,\BZ)$ a pseudo-Anosov homeomorphism whose
action on the first homology of the surface is the given matrix
\cite{Papadopoulos}.  In \cite{Clay-Pettet}, we use Theorem \ref{th:A}
and the techniques developed within this current paper, to construct
for any matrix in $\GL(k,\BZ)$ a fully irreducible hyperbolic element
whose action on the first homology of $F_k$ is the given matrix.

Consider the descending sequence of subgroups of $\Aut F_k$ given by   
\[ \Aut F_k \to \Aut(F_k/\Gamma^{i+1}(F_k)) \] where $\Gamma^2(F_k) =
[F_k,F_k]$, the commutator subgroup of $F_k$, and $\Gamma^{i+1}(F_k) =
[F_k,\Gamma^i(F_k)]$.
The {\it Johnson filtration} of $\Out F_k$ is the induced sequence  $J_k^1 \supset J_k^2 \supset \cdots$ of subgroups of $\Out F_k$; the group $J_k^1$ is analogous to the {\it Torelli subgroup} of the mapping class group of a surface. Observe that $[J_k^i,J_k^i] \subset J_k^{i+1}$, so that by applying
Theorem \ref{th:A} we have:

\begin{cor}\label{co:IA_k}
  For $k \geq 3$, there exist hyperbolic fully irreducible elements
  arbitrarily deep in the Johnson filtration for $\Out F_k$.
\end{cor}

To prove Theorem \ref{th:A}, we use methods necessarily very different
from Thurston's, which employed much of the rich geometry of
Teichm\"uller space. Our argument is based closely on an alternate,
more combinatorial proof due to Hamidi-Tehrani \cite{Hamidi-Tehrani}
that applies a variant on the usual ping pong argument to the
set of simple closed curves on a surface. Much of the work in our
paper is concerned with setting up a suitable substitute for the
intersection number of two simple closed curves on a surface, a key
ingredient in Hamidi-Tehrani's argument.


Observe that the intersection number between two curves $\alpha$ and
$\beta$ on a surface $S$ is equal to the combinatorial translation
length of the element $\alpha \in \pi_1(S)$ on the dual tree to lifts
of $\beta$ in the hyperbolic plane $\mathbb{H}^2$. This dual tree is
exactly the Bass--Serre tree for the splitting of the surface group
over the cyclic subgroup generated by $\beta$.  We formulate a
generalization of intersection numbers to finitely generated subgroups
$H$ of $F_k$.  In the following, $T^H$ denotes a minimal non-empty
$H$--invariant subtree of $T$.  %


\begin{defiV}
  Suppose $H$ is a finitely generated free group that acts on a
  simplicial tree $T$ such that the stabilizer of an edge is either
  trivial or cyclic. The \emph{free volume $\vol_T(H)$} of $H$ with
  respect to $T$ is the number of edges of the graph of groups
  decomposition $T^H/H$ with trivial stabilizer. 
\end{defiV}
\noindent It should be remarked that different notions of intersection
number have been developed by Scott--Swarup \cite{Scott-Swarup},
Guirardel \cite{Guirardel-number}, and Kapovich--Lustig
\cite{KL_Complex}, but that ours has been tailored to suit the needs
of our theorem.

The main ingredient in our proof of Theorem \ref{th:A} is the
following result concerning the growth of the free volume under
iterations of a Dehn twist:

\begin{thmB}
  Let $\delta_1$ be a Dehn twist associated to the very small cyclic
  tree $T_1$ with edge stabilizers generated by conjugates of the
  element $c_1$ and let $T_2$ be any other very small cyclic tree.
  Then there exists a constant $C = C(T_1,T_2)$ such that for any
  finitely generated malnormal or cyclic subgroup $H \subseteq F_k$
  with $\rank(H) \leq R$ and $n \geq 0$ the following hold:
  \begin{align*}
  \vol_{T_2}(\delta_1^{\pm n}(H)) &\geq 
  \vol_{T_1}(H) \bigl( n\ell_{T_2}(c_1) - C \bigr) - 
  M\vol_{T_2}(H)\\
  \vol_{T_2}(\delta_1^{\pm n}(H)) &\leq 
  \vol_{T_1}(H) \bigl( n\ell_{T_2}(c_1) + C \bigr) + 
  M\vol_{T_2}(H)
\end{align*}
where $M$ is the constant from Proposition \ref{prop:safe}.
\end{thmB}

\noindent The inequalities in Theorem \ref{th:B} should be compared
with the following inequality from \cite{FLP} (see also \cite{Ivanov})
for simple closed curves $\alpha, \beta$, and $\gamma$ on a surface:
\[ |i(\delta_\beta^{\pm n}(\gamma),\alpha) -
ni(\gamma,\beta)i(\alpha,\beta)| \leq i(\gamma,\alpha) \] Here
$i(\cdot,\cdot)$ denotes the geometric intersection number of two
simple closed curves, and $\delta_\beta$ is the Dehn twist around the
curve $\beta$.  An asymptotic version of Theorem \ref{th:B} for cyclic
subgroups appears as a special case of Cohen and Lustig's ``Skyscraper
Lemma'' \cite[Lemma 4.1]{Cohen-Lustig}.


Although it is not essential to our main theorem, we describe a
property of our notion of intersection number which likens it to
intersection number for surfaces, as we consider it of independent
interest. Recall that if $\alpha$ and $\beta$ are simple closed curves
that fill a surface $S$, and if $\sigma$ is any hyperbolic metric on $S$, %
then there is constant $K$ such that for any simple closed curve
$\gamma$ on $S$:
\begin{equation}\label{eq:bilip-surface}
  \frac{1}{K}\ell_\sigma(\gamma) \leq i(\alpha,\gamma) +
  i(\beta,\gamma) \leq K\ell_\sigma(\gamma)
\end{equation}
where $\ell_\sigma(\gamma)$ is the length of the geodesic representing
$\gamma$ with respect to the metric $\sigma$. 

Now recall that Culler--Vogtmann's Outer space $CV_k$ consists of
minimal discrete free actions of $F_k$ on $\BR$--trees, normalized
such that the sum of the lengths of the edges in the quotient graph is
1 \cite{Culler-Vogtmann}. A point of $CV_k$, or its unprojectivized
version $cv_k$, plays the role of a marked hyperbolic metric on
$S$. There is a compactification $\overline{CV}_k$
\cite{Culler-Morgan} which is covered by $\overline{cv}_k$.  The space
$\overline{cv}_k$ is the space of minimal \emph{very small} actions of
$F_k$ on $\BR$--trees \cite{Bestvina-Feighn,Cohen-Lustig}.  Kapovich
and Lustig showed that if $T_1$ and $T_2$ are trees in
$\overline{cv}_k$ that are ``sufficiently transverse'', then for any
tree $T \in cv_k$ there is a constant $K$ such that for any element $g
\in F_k$:
\begin{equation}\label{eq:bilip-free}
  \frac{1}{K}\ell_T(g) \leq \ell_{T_1}(g) + \ell_{T_2}(g) \leq
  K\ell_T(g)
\end{equation}
where $\ell_T(\cdot)$ is the translation length function for the tree
$T$.  We show a different generalization of \eqref{eq:bilip-surface}.

\begin{thmC}
  Let $T_1$ and $T_2$ be two very small cyclic trees for $F_k$ that
  fill and $T \in cv_k$. Then there is a constant $K$ such that for
  any proper free factor or cyclic subgroup $X \subset F_k$:
  \begin{equation*}
    \frac{1}{K}\vol_T(X) \leq \vol_{T_1}(X) + \vol_{T_2}(X) \leq K\vol_T(X).
  \end{equation*}
\end{thmC}

Our paper is organized as follows.  Section \ref{sc:prelim} recalls
well-known facts about $\Out F_k$ along with the definitions needed.
The only new material in this section is a discussion on ``filling''
cyclic trees.  In particular, we present a construction for producing
filling cyclic trees when $k \geq 3$.  In Section \ref{counting} we
describe how to compute the free volume of a finitely generated
subgroup of $F_k$ with respect to a cyclic tree.  This should be
compared to the ``no bigon'' condition for computing intersection
numbers between simple closed curves on a surface.  The main result of
Section \ref{sc:growth} is to give a proof of Theorem \ref{th:B}.  The
Hamidi-Tehrani ping pong argument is applied in Section
\ref{sc:pingpong} to prove Theorem \ref{th:A}. Finally, in Section \ref{sc:bilip} we
prove Theorem \ref{th:C}.

\bigskip


\noindent {\bf Acknowledgements.} The authors would like to thank Juan
Souto for helpful conversations, and for promoting this as an
interesting project. We also thank Mladen Bestvina for fielding our
questions, Michael Handel for pointing out an equivalent definition of
Definition \ref{filling} and for pushing us to remove the
``primitive'' hypothesis of the edge stabilizer from several of the
statements in an earlier version of this work, and the referee for
several helpful comments regarding this manuscript. The second author
is grateful to Yael Algom-Kfir, Vincent Guirardel, and Johanna
Mangahas for pleasant conversations related to this project.


\section{Preliminaries}\label{sc:prelim}

\subsection{Basics}\label{ssc:basics}


Let $F_k$ denote the rank $k$ non-abelian free group. For a basis $\CA
= \{ x_1, \ldots, x_k\}$ we fix a marked $k$--petaled rose $\Lambda =
\Lambda_\CA$: a graph with one vertex and $k$ oriented petals with an
identification to the set $\{ x_1, \ldots, x_k\}$, thus equipped with
an isomorphism $F_k \to \pi_1(\Lambda,\text{vertex})$. A marking of a
graph $\CG$ with $\pi_1(\mathcal{G}) \cong F_k$ is a homotopy
equivalence $\Lambda \to \CG$. An outer automorphism $\phi$ of the
free group determines a homotopy equivalence $\Phi\co \Lambda \to
\Lambda$. This gives a right action of $\Out F_k$ on the set of %
markings by precomposing the homotopy equivalence $\Lambda \to \CG$
by %
$\Phi$; that is, $\phi$ acts by changing the marking. The universal
cover of a marked graph $\CG$ is a tree $\tilde{\CG}$ equipped with a
free action of $F_k$; the set of such trees inherits the right action
of $\Out F_k$, which coincides with the action of $\Out F_k$ on Outer
space $CV_k$ or $cv_k$.


Given a simplicial map $f_0\co \CH_0 \to \CG$ between graphs, either
it is an immersion (i.e.,~locally injective), or there is some pair of
edges $e_1,e_2$ sharing a common initial vertex in $\CH_0$ that have
the same image under $f_0$. In case of the latter, let $\CH_1$ be the quotient graph of
$\CH_0$ obtained by identifying $e_1$ with $e_2$; then $f_0$ descends
to a well-defined map $f_1\co \CH_1 \to \CG$. We say that the map
$f_1: \CH_1 \to \CG$ is obtained from $f_0: \CH_0 \to \CG$ by a {\it
  fold}. Folding can be iterated until the resulting simplicial map
$f\co \CH \to \CG$ is an immersion of graphs \cite{Stallings}.  In the
case that $\CH$ has valence one vertices, adjacent edges can be iteratively pruned from $\CH$ to obtain a core graph $\CH_{core}$ (a graph
in which every edge belongs to at least one cycle) to which $f$
restricts to a map $f_{core}\co \CH_{core} \to \CG$.

Using folding, we can associate to the conjugacy class of a finitely
generated subgroup $H$ of $F_k$ an immersion of a core graph
$\CG^H_\CA \to \Lambda_\CA$.  Fix a basis for $H$, and let $\CH$ be a
$\rank(H)$--petaled rose, where each petal is subdivided into edges labeled according to the associated word in the basis $\CA$. The labels
determine a map $\CH \to \Lambda_\CA$; after a series of folds, the
resulting map is an immersion of graphs which can be pruned to obtain an
immersion of the core graph $\CG^H_\CA \to \Lambda_\CA$.  The
immersion $\CG^H_\CA \to \Lambda_\CA$ does not depend on the initial
graph $\CH$.  We refer to Stallings' paper \cite{Stallings} for more
details.

For a basis $\CA$ of $F_k$ and element $x \in F_k$, we let $|x|_\CA$
denote the reduced word length of $x$ with respect to the basis $\CA$.
When dealing with word length in free groups the following lemma due
to Cooper is indispensable:
\begin{lem}[Bounded cancellation \cite{Cooper}]\label{lem:BCC}
  Suppose $\CA_1$ and $\CA_2$ are bases for the free group $F_k$.
  There is a constant $C = C(\CA_1,\CA_2)$ such that if $w$ and $w'$
  are two elements of $F_k$ where:
  \[ |w|_{\CA_1} + |w'|_{\CA_1} = |ww'|_{\CA_1} \]
  then
  \[ |w|_{\CA_2} + |w'|_{\CA_2} - |ww'|_{\CA_2} \leq 2C. \]
\end{lem}

\noindent We denote by $BCC(\CA_1,\CA_2)$ the bounded cancellation
constant; that is, the minimal constant $C$ satisfying the lemma for
$\CA_1$ and $\CA_2$.  In other words, if $ww'$ is a reduced word in
$\CA_1$, $w = \prod_{i=1}^m x_i$ and $w' = \prod_{i=1}^{m'}x'_i$ where
$x_i,x'_i \in \CA_2$, then for $C = BCC(\CA_1,\CA_2)$ the subwords
$x_1 \cdots x_{m-C-1}$ and $x'_{C+1}\cdots x'_{m'}$ appear as subwords
of $ww'$ when considered as a word in $\CA_2$.


Besides the free simplicial $F_k$--actions arising from marked graphs,
we will also consider free group actions on simplicial trees that
arise as Bass-Serre trees of splittings of $F_k$ over cyclic
subgroups. In general, for an $F_k$--tree $T$ the action when
restricted to a finitely generated subgroup $H$ is not minimal, i.e.,
there is a non-empty proper $H$--invariant subtree.  When $H$ does not
fix a point in $T$, we let $T^H$ denote the smallest non-empty proper
$H$--invariant subtree of $T$.  Such a subtree is characterized as
the %
union of the axes of all of the elements of $H$ that do not fix a %
point in $T$ \cite{Culler-Morgan}.  When $H$ fixes a subtree of $T$ %
pointwise, we let $T^H$ be any point of $T$ fixed by $H$.  We denote
by $\ell_T(x)$ the translation length of the element $x \in F_k$ in
the tree $T$.


\subsection{Dehn twist automorphisms}\label{ssc:dehn}

The simplest type of homeomorphism of a surface is a \emph{Dehn
  twist}.  These homeomorphisms are supported on an annular
neighborhood of a simple closed curve and are defined by cutting the
surface open along the curve and regluing after twisting one side by
$2\pi$.  Algebraically, a simple closed curve on a surface $\alpha
\subset S$ determines a splitting of the fundamental group $\pi_1(S)$
either as an amalgamated free product $\pi_1 (S_1) *_{\la \alpha \ra}
\pi_1(S_2)$, if $\alpha$ is separating ($S - \alpha = S_1 \,
{\scriptstyle \coprod} \, S_2$); or as an HNN-extension
$\pi_1(S')*_{\la \alpha \ra}$, if $\alpha$ is nonseparating ($S -
\alpha = S'$).

By analogy, we now define a \emph{Dehn twist automorphism}; see
\cite{Cohen-Lustig, Levitt, RS} for their use in various other
settings.  First consider the splitting of $F_k = A *_{\la c \ra} B$
which expresses $F_k$ as an amalgamation of two free groups over a
cyclic group. Define an automorphism $\delta$ of $F_k$ by:
\begin{align*}
  \forall a \in A \qquad & \delta(a) = a \\
  \forall b \in B \qquad & \delta(b) = cbc^{-1}.
\end{align*}
The automorphism $\delta$ acts trivially on homology and therefore
belongs to the subgroup $\IA_k$. Dehn twist automorphisms arising from
amalgamations over $\BZ$ should be considered analogous to a Dehn
twist around a separating simple closed curve on a surface.

We similarly obtain an automorphism $\delta$ from an HNN-extension of
the form
$$F_k = A *_\BZ = \la A, t  \ | \ t^{-1}a_0t = a_1 \ra$$
for $a_0, a_1 \in A$ by:
\begin{align*}
  \forall a \in A \qquad & \delta(a) = a \\
  & \delta(t) = a_0t.
\end{align*}
Automorphisms arising from HNN-extensions should be compared to a Dehn
twist around a nonseparating curve on a surface.


As there is no way to orient a free group, we cannot speak of ``left
Dehn twists'' or ``right Dehn twists'' as for surfaces.  Thus when we
say ``$\delta$ is a Dehn twist associated to the cyclic tree $T$,'' we
are referring to one of the above defined automorphisms for the given
edge group.


From Bass--Serre theory, a splitting of $F_k$ over $\BZ$ defines an
action of $F_k$ on a tree $T$, the {\it Bass--Serre tree} of the
splitting (see \cite{Bass} or \cite{Serre}). We will refer to such
$F_k$--trees as {\it cyclic}.  In a certain sense, cyclic trees for
$F_k$ correspond to simple closed curves on a surface.  In particular,
Dehn twist automorphisms associated to cyclic trees generate
an index two subgroup of $\Aut F_k$ (the subgroup which induces an
action of $\SL_k(\BZ)$ on homology). Note that if $\delta$ is the Dehn
twist automorphism associated to the cyclic tree $T$, then $\delta$
preserves the action of $F_k$ on $T$, i.e.,~there is an isometry
$h_\delta\co T \to T$ such %
that $\forall g \in F_k$ and $\forall x \in T$ we have $h_\delta(gx)
= %
\delta(g)h_\delta(x)$.  In particular, $\ell_T(\delta(g)) =
\ell_T(g)$ %
for all $g \in F_k$.  %


We are primarily interested in the {\it outer} automorphism group of
$F_k$, and so in the sequel a Dehn twist will refer to an element of
$\Out F_k$ which is induced by a Dehn twist automorphism in $\Aut
F_k$.

\subsection{Guirardel's core and free volume}\label{ssc:core}

Our strategy for proving Theorem \ref{th:A} requires some notion of
intersection number between a cyclic tree $T$ and a free factor or
cyclic subgroup $X \subset F_k$. To motivate our choice of
intersection number we re-examine intersections of curves on surfaces.

For two simple closed curves $\alpha,\beta \subset S$, the
intersection number $i(\alpha,\beta) = \ell_{T_\alpha}(\beta)$ where
$T_\alpha$ is the Bass-Serre tree dual to the splitting of $\pi_1 (S)$
over $\alpha$.  Hence our notion of intersection number between a
cyclic tree $T$ and a cyclic group $X = \la g \ra$ should be equal to
$\ell_T(g)$. Given a subsurface $S_0 \subset S$ and a simple closed
curve $\alpha \subset S$, there is an obvious way to define an
intersection number $i(\alpha,S_0)$ by considering the boundary
$\partial S_0$ and setting $i(\alpha,S_0) = i(\alpha,\partial S_0)$
(when $\partial S_0$ is not connected we take the sum over the
individual components).  This is exactly twice the number of arc
components in $\alpha \cap S_0$.

Using the \emph{Guirardel core}, one can associate a ``subsurface'' to
a free factor relative to a pair of cyclic trees $T_1$ and $T_2$.  As
the Guirardel core is not used in later sections, we will not give the
complete definition; for more details see \cite{Guirardel-number} or
\cite{BBC}.  For our purposes we only need to know that the core $\CC
\subset T_1 \times T_2$ is an $F_k$--invariant subset (with respect to
the diagonal action), $\CC/F_k$ is a finite complex equipped with two
tracks representing the splittings associated to the cyclic trees
$T_1$ and $T_2$.  Further, the projection maps $T_1 \gets T_1 \times
T_2 \to T_2$ descend to maps $T_1/F_k \gets \CC/F_k \to T_2/F_k$.  The
tracks in $\CC/F_k$ are the preimages of the midpoints of the edges
$T_1/F_k$ and $T_2/F_k$.

Now to get a ``subsurface'' for a free factor $X \subset F_k$, we
restrict the actions on $T_1$ and $T_2$ to the subgroup $X$ and
consider the core $\CC^X \subset T_1^X \times T_2^X$.  The natural
inclusions $T_1^X \to T_1$ and $T_2^X \to T_2$ induce an inclusion
$\CC^X \to \CC$ and a ``subsurface inclusion'' map $\CC^X/X \to \CC/X
\to \CC/F_k$.  The key point is that $\CC^X/X$ is a finite complex
representing $X$.  The picture one should keep in mind is the
inclusion of the core of the cover of a subsurface into the cover
associated to the subsurface, as well as its image in the surface under the
covering map.  See Figure \ref{fg:core}.

\begin{figure}[t]
  \centering
  \includegraphics{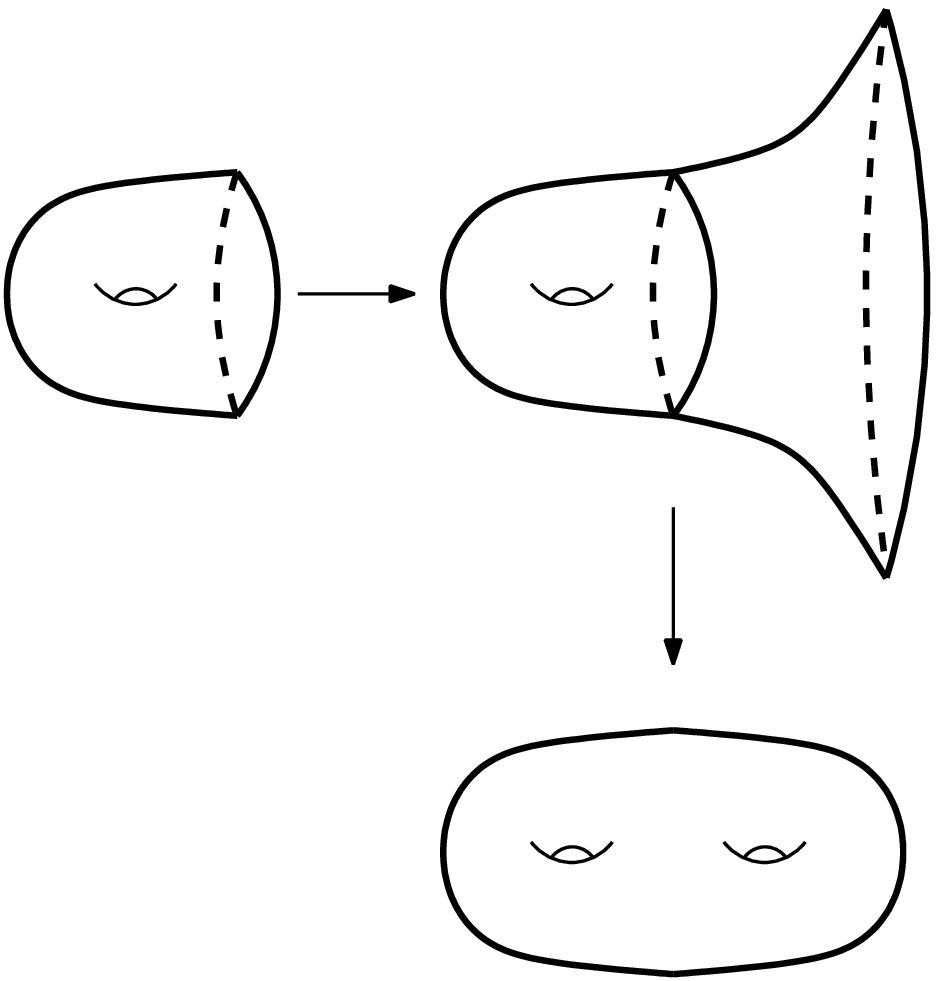}
  \caption{The map $\CC^X/X \to \CC/X \to \CC/F_k$.}\label{fg:core}
\end{figure}

Therefore, by analogy we should define the intersection number between
a cyclic tree $T_1$ and a free factor $X$ as the number of 
simply connected tracks associated to $T_1$ in $\CC^X/X$.  The map
$\CC^X/X \to T_1^X/X$ sends the simply connected tracks associated to
$T_1$ to edges of $T_1^X/X$ that have trivial edge stabilizer.  Thus
we define:

\begin{defi}[Free volume]\label{defi:volume}
  Suppose $X$ is a finitely generated free group that acts on a
  simplicial tree $T$ such that the stabilizer of an edge is either
  trivial or cyclic. The \emph{free volume $\vol_T(X)$} of $X$ with
  respect to $T$ is the number of edges in the graph of groups
  decomposition $T^X/X$ with trivial stabilizer.
\end{defi}
\noindent A form of this definition appears in \cite{Guirardel} in a more
general setting.  Notice that for a cyclic subgroup $X = \la g \ra$ we
have $\vol_T(X) = \ell_T(g) = \# $ edges of $T^X/X$, as we desired for
our notion of intersection.  If $T$ is equipped with a metric
preserved by the action of $X$, the free volume $\vol_T(X)$ is the sum
of lengths of the edges of $T^X/X$ with trivial edge stabilizer.
Clearly free volume only depends on the conjugacy class of the
subgroup.

\begin{lem}\label{lm:malnormal-freevolume}
  Let $T$ be a cyclic tree for $F_k$ and $X$ a subgroup of $F_k$.  If
  $X$ is malnormal, then there is at most one edge of $T^X/X$ that has
  a nontrivial stabilizer.  Hence, if $X$ is finitely generated, then:
  \[ 0 \leq \#\mbox{edges of } T^X/X - \vol_T(X) \leq 1. \]
\end{lem}

\begin{proof}
  Suppose that $T$ is the cyclic tree dual to the splitting of $F_k$
  over the cyclic subgroup $\la c \ra$.  For any subgroup $X \subseteq
  F_k$, edges of $T^X/X$ correspond to double cosets $Xg \la c \ra$
  and the corresponding subgroup is represented by the $X$--conjugacy
  class of $X \cap g\la c \ra g^{-1}$ \cite{Scott-Wall}.  Suppose both
  $X \cap g\la c \ra g^{-1}$ and $X \cap h\la c \ra h^{-1}$ are
  nontrivial for some $g,h \in \la c \ra$.  Thus for some $n \neq 0$,
  we have $gc^ng^{-1}, hc^nh^{-1} \in X$.  Therefore, $gc^ng^{-1} \in
  X \cap gh^{-1} X (gh^{-1})^{-1}$.  As $X$ is malnormal, this implies
  $gh^{-1} \in X$ and therefore the double cosets $Xg \la c \ra$ and
  $Xh \la c \ra$ are the same.  Hence, there is at most one edge of
  $T^X/X$ with a nontrivial stabilizer.
\end{proof}

\subsection{Filling cyclic trees}\label{ssc:filling}

Recall that two simple closed curves $\alpha$ and $\beta$ on a surface
are said to {\it fill} when the sum of their intersection numbers with
any arbitrary simple closed curve is positive. This naturally leads
one to consider the following definition.
\begin{defi}[Filling]\label{filling}
  We say that two cyclic trees $T_1$ and $T_2$ for $F_k$ \emph{fill}
  if
  \begin{equation*}\label{eq:F1}
    \vol_{T_1}(X) + \vol_{T_2}(X) > 0\tag{\textbf{F1}}
  \end{equation*}
  for every proper free factor or cyclic subgroup $X \subset F_k$.
\end{defi}

We thank Michael Handel for pointing out the follow lemma.  The lemma
says that the above condition (\textbf{F1}) is the minimal hypothesis
one could impose on cyclic trees in Theorem \ref{th:A}.

\begin{lem}\label{Handel}
  Suppose $T_1, \, T_2$ are cyclic trees for $F_k$ and $\delta_1, \,
  \delta_2$ are the respective Dehn twist (outer) automorphisms.  Then
  $T_1$ and $T_2$ satisfy {\rm (\textbf{F1})} if and only if no
  conjugacy class of a proper free factor or cyclic subgroup of $F_k$
  is invariant under both $\delta_1$ and $\delta_2$.
\end{lem}

\begin{proof}
  Suppose $X$ is a cyclic subgroup or proper free factor of $F_k$ such
  that $\vol_{T_i}(X) = 0$ for $i = 1$ or 2.  Thus either $X$ is
  contained in a vertex stabilizer of $T_i$, or else $T_i^X/X$ is a
  single edge with a nontrivial stabilizer represented by an edge
  stabilizer of $T_i$.  In the former case, it is clear that
  $\delta_i$ fixes the conjugacy class of $X$; in the latter case,
  conjugating $\delta_i$ by an appropriate inner automorphism results
  in an automorphism whose action on $X$ agrees with the action of the
  Dehn twist automorphism of $T^X$.  Thus the conjugacy class of $X$
  is invariant under $\delta_i$.  Hence if $\vol_{T_1}(X) +
  \vol_{T_2}(X) = 0$, then the conjugacy class of $X$ is invariant
  under both $\delta_1$ and $\delta_2$.

  For the converse, it follows from Theorem \ref{th:B} that if $X$ is
  a cyclic subgroup or proper free factor such that $\vol_{T_i}(X)
  \neq 0$ for $i = 1$ or 2, then the conjugacy class of $X$ is not
  fixed by $\delta_i$.  Thus if $\vol_{T_1}(X) + \vol_{T_2}(X) > 0$,
  then the conjugacy class of $X$ is not invariant under both
  $\delta_1$ and $\delta_2$.
\end{proof}


Now recall that for surfaces we have the following equivalent
definitions of filling curves: (1) two curves fill if the complement of
their union is a union of topological disks, and; (2) two curves fill if
no proper subsurface contains the union of the curves.  Each of these 
characterizations leads to an alternative notion for two cyclic
trees $T_1$ and $T_2$ to fill:

\medskip
\noindent (\textbf{F2}) $F_k$ acts freely on the product $T_1 \times
T_2$, i.e., no element of $F_k$ fixes a point in each tree.


\noindent(\textbf{F3}) If $c_1$ fixes an edge of $T_1$ and $c_2$ fixes
an edge of $T_2$, then the subgroup $\la c_1, c_2 \ra$ is not
contained in a proper free factor of $F_k$.


\medskip The advantage of these alternate conditions is that
(\textbf{F2}) can be checked using Stallings' graph pull-backs
\cite{Stallings}, and (\textbf{F3}) can be checked using a version of
Whitehead's algorithm (see for instance \cite{Algom-Kfir} or
\cite{Martin}).  Obviously (\textbf{F1}) implies (\textbf{F2}), and
while some of the other relations are not clear, we will show that
(\textbf{F2}) + (\textbf{F3}) implies (\textbf{F1}). In a later
example we will see that (\textbf{F3}) is not implied by (\textbf{F1})
+ (\textbf{F2}).

\begin{prop}\label{prop:fill}
  Suppose $T_1$ and $T_2$ are cyclic trees satisfying {\rm
    (\textbf{F2})} and {\rm (\textbf{F3})}.  Then the trees $T_1$ and
  $T_2$ fill, i.e., $T_1$ and $T_2$ satisfy {\rm (\textbf{F1})}.
\end{prop}

\begin{proof}
  As (\textbf{F2}) implies that no $g \in F_k$ fixes a point in both
  $T_1$ and $T_2$, clearly $\vol_{T_1}(\la g \ra) + \vol_{T_2}(\la g
  \ra) > 0$ for any $g \in F_k$.


  Now suppose that $X$ is a proper free factor such that
  $\vol_{T_1}(X) + \vol_{T_2}(X) = 0$.  If $X$ fixes a vertex in $T_1$
  then $X$ must act freely on $T_2$ by (\textbf{F2}) and hence
  $\vol_{T_2}(X) > 0$.  Similarly $\vol_{T_1}(X) > 0$ if $X$ fixes a vertex in $T_2$.
  Therefore we can assume that $X$ does not fix a vertex in both $T_1$
  and $T_2$. As $X$ is a free factor and hence malnormal, by
  Lemma \ref{lm:malnormal-freevolume}, the only way $\vol_{T_1}(X) +
  \vol_{T_2}(X) = 0$ is if both quotient graphs of groups $T_1^X/X$
  and $T_2^X/X$ consist of a single edge with a nontrivial stabilizer.
  This contradicts (\textbf{F3}).  Therefore $\vol_{T_1}(X) +
  \vol_{T_2}(X) > 0$ and hence $T_1$ and $T_2$ fill.

\end{proof}

We can use this Proposition to produce filling cyclic trees.

\begin{ex}\label{ex:fill}
  Let $T$ be the cyclic tree for $F_3$ dual to the
  splitting $F_3 = \la a,c\ra*_{\la c \ra}\la b,c \ra$ and let $\phi
  \in \Out F_3$ be the element represented by $a \mapsto b
  \mapsto c \mapsto ab$.  We claim that the cyclic trees $T$
  and $T\phi^{-6}$ fill.  For reference we make note of $\phi^6$:
  \begin{align*}
    a & \mapsto abbc \\
    \phi^6\co \quad b & \mapsto bccab \\
    c & \mapsto cababbc
  \end{align*}
  Vertex stabilizers of $T\phi^{-6}$ are conjugates of $\la abbc ,
  cababbc \ra$ and $\la bccab , cababbc \ra$.  Using pull-back
  diagrams it is easy to see that the intersections of the vertex
  stabilizers are empty.  Hence the trees $T$ and $T\phi^{-6}$ satisfy
  (\textbf{F2}) and therefore $\vol_{T}(\la g \ra) +
  \vol_{T\phi^{-6}}(\la g \ra) > 0$ for any $g \in F_k$.

  Unfortunately, the trees $T$ and $T\phi^{-6}$ do not satisfy
  (\textbf{F3}) as $\langle c, cababbc \rangle$ is a proper free
  factor of $F_3$ ($F_3 = \la c,cababbc\ra *\la ab\ra$).  We can show
  that essentially this is the only such proper free factor, and that
  it satisfies (\textbf{F1}).

  Suppose that $X$ is a proper free factor that contains $\la c_1,c_2
  \ra$ where $c_1$ fixes an edge of $T$ and $c_2$ fixes an edge of
  $T\phi^{-6}$. Then by replacing $X$ by a conjugate, we can assume
  that $X = \la c, g\phi^6(c)g^{-1} \ra$ for some $g \in F_k$.
  However, it is easy to see that $\vol_T(X) \geq 3$ for this subgroup
  as the translation length of $\phi^6(c)$ in $T$ is 4.  Other proper
  free factors satisfy (\textbf{F1}) by the argument in Proposition
  \ref{prop:fill}.  Hence $\vol_T(X) + \vol_{T\phi^{-6}}(X) > 0$ for
  any proper free factor and therefore $T$ and $T\phi^{-6}$ fill.

\end{ex}

To build filling cyclic trees in arbitrarily high rank we introduce
two simplicial complexes naturally associated to $F_k$; these
complexes appear in \cite{KL_Complex}. They are analogous to the curve
complex for the mapping class group, i.e.,~the simplicial complex
whose vertices are isotopy classes of simple closed curves and
simplicies correspond to disjoint representatives.

The \emph{dominance graph} $\calD$ is the graph whose vertices
correspond to conjugacy classes of proper free factors of $F_k$, where
two such $[A]$ and $[B]$ are connected by an edge if there are
representatives, $A' \in [A]$, $B' \in [B]$, with $A' \subset B'$ or
$B' \subset A'$.  This is the 1--skeleton of the free factor complex
considered by Hatcher and Vogtmann \cite{HV}.

We also consider the \emph{cyclic splitting graph} $\CZ'$, although
what we actually require is the following variant of the like-named
complex appearing in \cite{KL_Complex}: Vertices correspond to very
small simplicial trees for $F_k$, i.e., simplicial trees $T$ such that
edge stabilizers are either trivial or maximal cyclic in adjacent
vertex stabilizer, and the stabilizer of any tripod is trivial.
Notice that cyclic trees where the edge stabilizers are generated by
primitive elements (i.e.,~can be extended to a basis) are vertices in
this graph. Two very small simplicial trees $T_1$ and $T_2$ are
adjoined by an edge in $\CZ'$ if there is a $g \in F_k$ such that
$\ell_{T_1}(g) = \ell_{T_2}(g) = 0$, i.e., $g$ fixes a point in both
$T_1$ and $T_2$.

The following proposition should now be compared to the fact that two
curves fill if and only if their distance in the curve complex is at
least 3.

\begin{prop}\label{prop:dominance}
  Suppose that $T_1$ and $T_2$ are cyclic trees with primitive cyclic
  edge generators $c_1$ and $c_2$, respectively, such that
  $d_{\CZ'}(T_1,T_2) \geq 2$ and $d_{\calD}([c_1],[c_2]) \geq 3$.
  Then the cyclic trees $T_1$ and $T_2$ fill.
\end{prop}

\begin{proof}
  Since $d_{\CZ'}(T_1,T_2) \geq 2$ there is no element $g \in F_k$
  such that $\ell_{T_1}(g) = \ell_{T_2}(g) = 0$, hence the trees $T_1$
  and $T_2$ satisfy (\textbf{F2}).  Further since
  $d_{\calD}([c_1],[c_2]) \geq 3$ there is no proper free factor $X
  \subset F_k$ or conjugates $c'_1 \in [c_1]$ and $c'_2 \in [c_2]$
  such that $\la c'_1, c'_2 \ra \subseteq X$, hence the trees $T_1$
  and $T_2$ satisfy (\textbf{F3}).  Therefore by Proposition
  \ref{prop:fill} the cyclic trees $T_1$ and $T_2$ fill.
\end{proof}

\begin{rmk}\label{rm:fill}
  For $k \geq 3$, Kapovich and Lustig have shown that for a hyperbolic
  fully irreducible element $\phi \in \Out F_k$ and any two vertices
  $[A],[B] \in \calD$ that $d_\calD([A],\phi^n([B]))$ goes to infinity
  as $n \to \pm \infty$ \cite{KL_Complex}.  Similarly for two vertices
  $T_1, T_2 \in \CZ'$.  Hence Proposition \ref{prop:dominance} shows
  for any cyclic tree $T$ whose edge stabilizers are generated by
  conjugates of a primitive element and any hyperbolic fully
  irreducible element $\phi \in \Out F_k$, that given sufficiently
  large $n$, the pair $T$ and $T\phi^n$ fill.
\end{rmk}

\section{Computing free volume}\label{counting}

In this section, we will explain how we use Stallings' folding to find
the free volume of finitely generated subgroups of $F_k$ relative to
cyclic trees. This will be central to our proof of Theorem \ref{th:B}.

\subsection{\texorpdfstring{Cyclic splittings of $F_k$}{Cyclic
    splittings of F_k}}\label{ssc:splittings}
We begin by recalling two theorems which describe how any cyclic
splitting of $F_k$ must arise. For the case of amalgamations, we have
the following theorem of Shenitzer:
\begin{thm}[Shenitzer \cite{Shenitzer}]\label{splitting1}
  Suppose that $F_k$ is expressed as an amalgamated free product
  $F_k=A *_{\la c \ra} B$. Then one of the following symmetric
  alternatives holds:
  \begin{enumerate}
  \item $A *_{\la c \ra} B = A *_{\la c \ra} \la c, B_0 \ra $ with
    $F_k = A * B_0$; or
  \item $A *_{\la c \ra} B = \la A_0, c \ra *_{\la c \ra} B$ with $F_k
    = A_0 * B$. \qed
  \end{enumerate}
\end{thm}
\noindent Interchanging $A \leftrightarrow B$ we will always assume
the first alternative holds.  Consequently, a Dehn twist automorphism 
$\delta$ resulting from a splitting of $F_k$ as an amalgamation over
$\BZ$ as above always arises as follows: There is a free splitting
$F_k = A * B_0$ and an element $c \in A$ such that:
\begin{align*}
  \forall a \in A \qquad & \delta(a) = a \\
  \forall b \in B_0 \qquad & \delta(b) = cbc^{-1}.
\end{align*}
A basis for $F_k$ \emph{relative to the cyclic tree} dual to $A*_{\la
  c \ra} B$ consists of the union of a basis for $A$ and a basis for
$B_0$.

There is an analogous theorem for HNN-extensions due to Swarup
\cite{Swarup}.
\begin{thm}[Swarup \cite{Swarup}]\label{th:HNN}
  Suppose that $F_k$ is expressed as an $HNN$-extension $F_k=A
  *_\BZ$. Express $F$ in terms of $A$ and an extra generator $t$, such
  that the edge group $\la c \ra = A \cap tAt^{-1}$. Then $A$ has a
  free product structure $A = A_1 * A_2$, in such a way that one of
  the following symmetric alternatives holds:
\begin{enumerate}
\item $\la c \ra \subset A_1$, and there exists $a \in A$ such that
  $t^{-1}\la c \ra t = a^{-1} A_2 a$; or
\item $t^{-1}\la c \ra t \subset A_1$, and there exists $a \in A$ such
  that $\la c \ra = a^{-1} A_2 a$. \qed
\end{enumerate}
\end{thm}
\noindent For alternative viewpoints and proofs see
\cite{Bestvina-Feighn,Louder,Stallings2}.  For our purposes we record
the following restatement of Theorem \ref{th:HNN}.

\begin{cor}\label{HNN-basis}\label{splitting2}
  Suppose that $F_k$ is expressed as an $HNN$-extension $F_k=A *_\BZ$.
  Then $F_k$ has a free product decomposition $F_k = A_0 * \la t_0
  \ra$ and $A$ has a free product decomposition $A = A_0 * \la
  t_0^{-1} c t_0 \ra$ for some $c \in A_0$.  Either $t = t_0a$ (case
  {\rm (1)} in Theorem \ref{th:HNN}), or $t = a^{-1}t_0^{-1}$ (case
  {\rm (2)}).\qed
\end{cor}
\noindent Again, by interchanging $A \leftrightarrow tAt^{-1}$ we will
always assume that first alternative holds.  Thus any Dehn twist
automorphism $\delta$ resulting from an HNN-extension over $\BZ$ as
above always arises as follows: There is a free splitting $F_k = A_0
* \la t_0 \ra$ and an element $c \in A_0$ such that:
\begin{align*}
  \forall a \in A_0 \qquad & \delta(a) = a \\
  & \delta(t_0) = ct_0.
\end{align*}
A basis for $F_k$ \emph{relative to the cyclic tree} dual to $A*_\BZ$
consists of the union of a basis for $A_0$ and $t_0$.

\begin{rmk}\label{remark:verysmall}
  We will usually restrict our attention to very small cyclic trees.
  This does not result in any loss of generality as any Dehn twist
  automorphism is a power of Dehn twist automorphism associated to a
  very small cyclic tree.  Indeed, suppose $T$ is a cyclic tree dual
  to an amalgamated free product that is not very small.  By the
  above, this tree is dual to a splitting $A *_{\la c^n \ra} \la
  c^n,B_0\ra$ where $c \in A$ is an indivisible element.  The
  associated Dehn twist is the $n^{\rm th}$ power of the Dehn twist
  associated to the very small cyclic tree dual to the splitting $A
  *_{\la c \ra} \la c, B_0 \ra$.
\end{rmk}
\subsection{Free volume for an amalgamated free
  product}\label{ssc:amalgamated}


Here we explain how to compute free volume for a finitely generated
subgroup $H$ with respect to a tree dual to an amalgamated product by
associating a tree with a {\it free} $F_k$--action, using Shenitzer's
Theorem.  We consider a splitting of $F_k$ as an amalgamated free
product of the form:
\[ F_k = A *_{\langle c \rangle} \langle c, B_0 \rangle \] with $F_k =
A * B_0$ and $c \in A$ with $c$ indivisible. Let $\CA = \{a_1, \ldots,
a_j \}$ be a basis for $A$, and $\CB_0 = \{ b_{j+1}, \ldots b_k \}$ a
basis for $B_0$.  We assume that $c$ is cyclically reduced with
respect to $\CA$.  Thus $\CA \cup \CB_0$ is a %
basis for $F_k$ relative to $T$.  Let $\Lambda = \Lambda_{\CA \cup %
  \CB_0}$ be the $k$--rose labeled by the basis $\CA \cup \CB_0$.
Then let $\Lambda_\CA$ be the $j$--rose, labeled by the elements of
$\CA$, let $\Lambda_{\CB_0}$ be the $(k-j)$--rose, labeled by the elements
of $\CB_0$, and let $\Lambda_\CB$ be the $(k-j+1)$--rose resulting from
wedging an additional circle corresponding to the element $c$ to
$\Lambda_{\CB_0}$.  There are natural inclusions $\iota_\CA\co
\Lambda_\CA \to \Lambda$ and $\iota_{\CB_0} \co \Lambda_{\CB_0} \to
\Lambda$ and an immersion $\iota_\CB \co \Lambda_\CB \to \Lambda$.  We
say that an edge of $\Lambda$ corresponding to an element of $\CA$ is
an $\CA$--edge and an edge of $\Lambda$ corresponding to an element of
$\CB_0$ is a $\CB_0$--edge.

Let $\tilde{\Lambda}_\CA$ and $\tilde{\Lambda}_\CB$ be the universal
covers of $\Lambda_\CA$ and $\Lambda_\CB$ respectively.  The covering
maps naturally define immersions $\tilde{\iota}_\CA \co
\tilde{\Lambda}_\CA \to \Lambda$ and $\tilde{\iota}_\CB \co
\tilde{\Lambda}_\CB \to \Lambda$. Let $\CV(\CA)$ denote the set of
subtrees of $\tilde{\Lambda}$ which are lifts of $\tilde{\iota}_\CA
\co \tilde{\Lambda}_\CA \to \Lambda$ to $\tilde{\Lambda}$, and let
$\CV(\CB)$ denote the set of subtrees of $\tilde{\Lambda}$ which are
lifts of $\tilde{\iota}_\CB \co \tilde{\Lambda}_\CB \to \Lambda$ to
$\tilde{\Lambda}$.  In other words, $\CV(\CA)$ is the set of minimal      %
subtrees in $\tilde{\Lambda}$ for conjugates of $A$, and similarly for     %
$\CV(\CB)$.

Notice that the natural simplicial structure (i.e.,~vertices and edges)
the trees in $\CV(\CB)$ inherit from $\tilde{\Lambda}_\CB$ is
different from the induced simplicial structure on the trees when
considered as subtrees of $\tilde{\Lambda}$, unless $c$ is primitive
and $c \in \CA$.  When we speak of ``intersection'' of these subtrees,
we will usually refer to the natural inherited simplicial structure.
To make this easier on the reader, we include the following
definition.

\begin{defi}\label{def:ocap}
  Suppose $X$ is a subtree of $\tilde{\Lambda}$ (considered with the
  induced simplicial structure) and $L$ is a subtree in $\CV(\CB)$
  (considered with the inherited simplicial structure).  We define
  the $\ocap$--intersection of $X$ and $L$, denoted $X \ocap L$ by the
  following:
  \[x \in X \ocap L \Leftrightarrow \begin{array}{c} x \in e \subset X
    \cap L, \mbox{ where $e$ is a union of edges in $X$} \\ \mbox{and
      a union of edges of $L$.}\end{array} \]
\end{defi}


There is an $F_k$--equivariant one-to-one correspondence between the
set $\CV(\CA) \cup \CV(\CB)$ and the set of vertices of $T$, defined
by common stabilizer subgroups in $F_k$. Two vertices in $T$ are
adjacent if and only if the $\ocap$--intersection of their
corresponding subtrees in $\CV(\CA)$ and $\CV(\CB)$ is nonempty and
hence an infinite line.  Thus we have a description of $T$ in terms of
intersection of subtrees of $\tilde{\Lambda}$ associated to $A$ and
$B$.


Recall that $H$ is a finitely generated subgroup of $F_k$, and that 
$\tilde{\Lambda}^H$ denotes the smallest $H$--invariant subtree of
$\tilde{\Lambda}$.  We seek to describe $T^H/H$ (and hence compute
$\vol_T(H)$) in terms of $\tilde{\Lambda}^H/H$ with additional data
encoding the edge types.  A subtree is \emph{trivial} if it is a
single vertex, otherwise it is \emph{nontrivial}.  We feature two sets
of nontrivial subtrees of $\tilde{\Lambda}^H$:

\begin{enumerate}
\item Nontrivial subtrees of the form $K^H = \tilde{\Lambda}^H \cap K$
  for $K \in \CV(\CA)$ which are not properly contained within a
  subtree $\tilde{\Lambda}^H \ocap L$ for $L \in \CV(\CB)$.  We denote
  by $\CV^H(\CA)$ the set of all such subtrees $K^H$.
\item Nontrivial subtrees of the form $L^H = \tilde{\Lambda}^H \ocap L$
  for $L \in \CV(\CB)$ which are not properly contained within a
  subtree $\tilde{\Lambda}^H \cap K$ for $K \in \CV(\CA)$. We denote
  by $\CV^H(\CB)$ the set of all such subtrees $L^H$.
\end{enumerate}


Notice that $\CV^H(\CA)$ is empty if and only if $H$ is contained in a
conjugate of $B$ so that $H$ fixes a vertex of $T$.  Similarly,
$\CV^H(\CB)$ is empty if and only if $H$ is contained a conjugate of
$A$.  Thus both $\CV^H(\CA)$ and $\CV^H(\CB)$ are empty if and only if
$H$ is contained in a conjugate of $\la c \ra$.  In either of these
cases the minimal tree $T^H$ is a single point and $\vol_T(H) = 0$.

For each subtree $K^H \in \CV^H(\CA)$ we have a corresponding vertex
$v_K \in T$ (the vertex corresponding to $K \in \CV(\CA)$, where $K
\cap \tilde{\Lambda}^H = K^H$); denote the set of such vertices by
$V^H(\CA)$. Likewise, for each component of $L^H \in \CV^H(\CB)$ there
is a corresponding vertex $v_L \in T$; denote the set of such vertices
by $V^H(\CB)$. Note that this correspondence between components of
$\CV^H(\CA) \cup \CV^H(\CB)$ and vertices of $T$ is $H$--equivariant
as $\tilde{\Lambda}^H$ is $H$--equivariant.


Let $\CE^H(\CA,\CB)$ denote the set of nonempty (but possibly trivial,
i.e., a vertex) subtrees $K^H \cap L^H$ for $K^H \in \CV^H(\CA)$ and
$L^H \in \CV^H(\CB)$. To each such subtree $K^H \cap L^H$ in
$\CE^H(\CA,\CB)$ is associated a (geometric) edge $e_{K}^{L}$ in $T$,
namely the edge with vertices $v_K$ and $v_{L}$.  Indeed as the          %
corresponding subtrees $K \in \CV(\CA)$ and $L \in \CV(\CB)$ intersect  %
(necessarily along an axis for some conjugate of $c$), the corresponding      %
conjugates of $A$ and $B$ fix a common edge of $T$; this edge is        %
$e_K^L$.  We denote the set of such edges by $E^H(\CA,\CB)$. The        %
correspondence between $\CE^H(\CA,\CB)$ and $E^H(\CA,\CB)$ is of        
course $H$--equivariant.                                                 


\begin{ex}\label{ex:mintree}
  It is perhaps enlightening at this point to consider an example of the sets
  and subtrees described above.  Let $T$ be the cyclic tree dual to the
  splitting of $F_3 = \la a_1,a_2,b \ra$ as $\la a_1, a_2 \ra *_{\la
    a_1a_2 \ra} \la a_1a_2,b \ra$ and $H = \la a_1b \ra$; the subtree 
  $\tilde{\Lambda}^H$ is the axis of $a_1b$. Then for $K \in
  \CV(\CA)$, the subtrees $\tilde{\Lambda}^H \cap K$ are the edges of
  $\tilde{\Lambda}^H$ that are labeled by $a_1$.  Similarly, the
  subtrees $\tilde{\Lambda}^H \cap L$ for $L \in \CV(\CB)$ are the
  segments consisting of edges labeled by $ba_1$.  With the inherited
  simplicial structure on a subtree $L$, only the edges on the axis
  of $a_1b$ labeled by $b$ are actually edges of $L$.  Thus, for $L
  \in \CV(\CB)$, the $\ocap$--intersections, $\tilde{\Lambda}^H \ocap
  L$ are the edges of $\tilde{\Lambda}^H$ labeled $b$.  It is easy to
  see that $\CE^H(\CA,\CB)$ is the set of vertices of
  $\tilde{\Lambda}^H$.  See Figure \ref{fig:mintree} for corresponding
  edges $E^H(\CA,\CB)$.  Notice that these edges form the axis of $H$ in
  $T$, which is $T^H$.
  \begin{figure}[t]
    \psfrag{a}{$a_1$}
    \psfrag{b}{$b$}
    \psfrag{A}{$\la a_1,a_2 \ra$}
    \psfrag{B}{$a_1b\la a_1,a_2 \ra(a_1b)^{-1}$}
    \psfrag{C}{$a_1\la a_1a_2, b \ra a_1^{-1}$}
    \psfrag{D}{$a_1ba_1\la a_1a_2, b \ra (a_1ba_1)^{-1}$}
    \psfrag{u}{$u$}
    \psfrag{v}{$v$}
    \psfrag{e}{$e$}
    \psfrag{f}{$a_1e$}
    \psfrag{g}{$a_1be$}
    \psfrag{h}{$a_1ba_1e$}
    \psfrag{w}{$a_1v$}
    \psfrag{y}{$a_1bu$}
    \psfrag{x}{$a_1ba_1v$}
    \centering
    \includegraphics{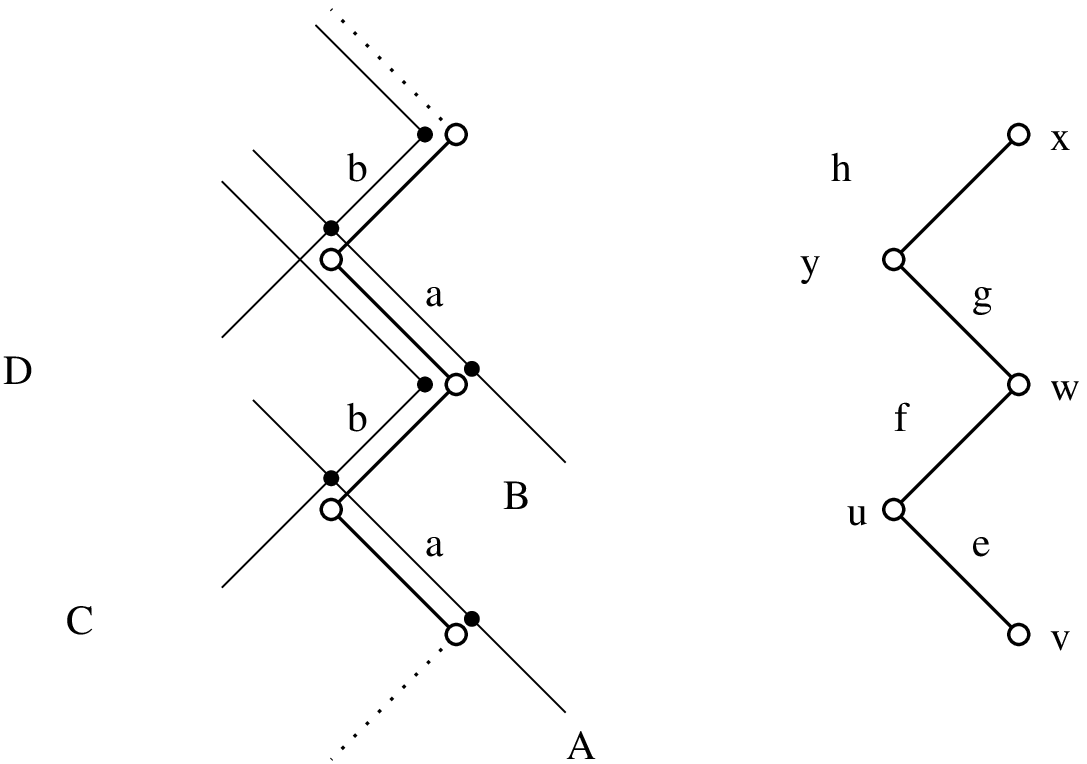}
    \caption{The edges $E^H(\CA,\CB)$ for Example \ref{ex:mintree}.
      On the left is the axis $\tilde{\Lambda}^H$ with the subtrees in
      $\CV(\CA)$ and $\CV(\CB)$ schematically draw in.  The inherited
      simplicial structure on these subtrees is shown.  On the right
      is the picture in $T$.  The vertices $u$ and $v$ are stabilized
      by $\la a_1,a_2 \ra$ and $\la a_1a_2, b \ra$, respectively, and
      the edge $e$ is stabilized by $\la a_1a_2 \ra$.}
    \label{fig:mintree}
  \end{figure}
\end{ex}


\begin{lem}\label{lm:minimaltree}
  Suppose $H$ does not fix a point in $T$.  Then the subcomplex in $T$
  consisting of vertices $V^H(\CA) \cup V^H(\CB)$ and edges
  $E^H(\CA,\CB)$ is precisely the smallest $H$--invariant subtree $T^H$
  of $T$.
\end{lem}
\begin{proof}
  Suppose that $v_K$ and $v_L$ are two vertices in $V^H(\CA) \cup
  V^H(\CB)$. Then there exists an arc in $\tilde{\Lambda}^H$ which
  connects the component $K$ to the component $L$.  This arc passes
  through a sequence of subtrees $K = K_0, K_1, \ldots, K_n = L \in
  \CV^H(\CA) \cup \CV^H(\CB)$.  As the arc transitions from $K_{i-1}$
  to $K_i$, the intersections $K_{i-1} \cap K_i$ are non-empty and
  therefore correspond to edges $e_i = e_{K_{i-1}}^{K_i} \in
  E^H(\CA,\CB)$.  By construction the edge path $e_1,\ldots,e_n$
  connects $v_K$ to $v_L$.  Therefore the subcomplex consisting of
  vertices $V^H(\CA) \cup V^H(\CB)$ and edges $E^H(\CA,\CB)$ is
  connected and hence an $H$--invariant subtree of $T$.


  To prove minimality, note that every edge $e$ in $E^H(\CA,\CB)$ lies   %
  on the axis of some element of $H$ acting on $T$.  Indeed, suppose
  $e$ corresponds to $K \cap L \in \CE^H(\CA,\CB)$ with $K \in
  \CV^H(\CA)$ and $L \in \CV^H(\CB)$.  Since $K$ is not contained in
  $L$ there is an $x \in K - (K \cap L)$.  Let $h \in H$ be such that
  the edge path from $x$ to $hx$ is contained in the axis of $h$ and
  the edge path from $x$ to $hx$ contains $K \cap L$.  Such an element
  exists since the action of $H$ on $\tilde{\Lambda}^H$ is minimal.
  Moreover, by choosing the axis of $h$ to be ``complicated enough''     %
  with respect to $\CV(\CA)$ and $\CV(\CB)$, i.e., the axis intersects   %
  several of these subtrees, we can assume that $h$ does not fix a       %
  point of $T$.  Notice that the axis of $h$ in $T$ contains $e$.  It    %
  is well-known that when a group acts on a tree without a global
  fixed point, the minimal tree is precisely the union of the axes of
  its elements \cite{Culler-Morgan}.
\end{proof}


We introduce some terminology which will be useful for classifying the
subtrees in $\CV^H(\CA)$, $\CV^H(\CB)$, and $\CE^H(\CA,\CB)$. Fix an
immersion $\gamma\co [0,1] \to \Lambda$ that factors through $[0,1]
\to S^1 \to \Lambda$, where the first map identifies $0$ and $1$, and
the second map represents the conjugacy class of $c \in F_k \cong
\pi_1(\Lambda)$.  We let $\Lambda^H$ be the graph
$\tilde{\Lambda}^H/H$. A {\it chain} is an ordered set $\alpha =
(\gamma_1, \ldots,\gamma_\ell)$, where $\gamma_i$ is a lift of
$\gamma$ to $\Lambda^H$, with $\gamma_i(1) = \gamma_{i+1}(0)$
for $i = 1, \ldots, \ell-1$.  The {\it vertices} of a chain are
$\mathcal{V}(\alpha) = \gamma_1(0) \cup \bigcup_{i=1}^\ell
\gamma_{i}(1)$.  Notice that vertices of a chain are vertices of
$\Lambda^H$, but vertices contained in the image of $\alpha$ are not
necessarily vertices of the chain.  We often identify a chain with its
image in $\Lambda^H$.

We refer to an edge in $\Lambda^H$ as an $\CA$--edge or $\CB_0$--edge
according to its image in $\Lambda$.  A chain $\alpha$ is {\it
  nonessential} if
\begin{enumerate}
\item any edge adjacent to $\alpha$ is a $\CB_0$--edge which is
  adjacent to $\alpha$ at a vertex in $\CV(\alpha)$; or
\item the only edges adjacent to $\alpha$ are $\CA$--edges.
\end{enumerate}

Otherwise we say $\alpha$ is {\it essential}.  In other words,
$\alpha$ is essential only if it is adjacent to a $\CB_0$--edge at a
vertex of $\Lambda^H$ that is not a vertex of $\alpha$, or it is
adjacent to both an $\CA$--edge and a $\CB_0$--edge.  The edges in a
nonessential chain adjacent only to $\CB_0$--edges are considered
$\CB_0$--edges.  The set of all maximal essential chains in $\Lambda^H$
is denoted by $\alpha(\Lambda^H)$.

We say a vertex is {\it essential} if it is not a chain vertex of any
essential chain and it is adjacent to both an $\CA$--edge and a
$\CB_0$--edge.  The set of all essential vertices we denote by
$\CV_{ess}(\Lambda^H)$.

\begin{lem}\label{lm:chains}
  With the notation above, the image of a subtree in $\CE^H(\CA,\CB)$
  in $\Lambda^H$ is either a maximal essential chain or an essential
  vertex.  Conversely, every maximal essential chain or essential
  vertex is the image of some subtree in $\CE^H(\CA,\CB)$.
\end{lem}

\begin{proof}
  Let $K \in \CV^H(\CA)$ and $L \in \CV^H(\CB)$ and suppose $K \cap L$
  is nonempty.  First suppose $K \cap L$ is a vertex.  Hence its image
  in $\Lambda^H$ is adjacent to both an $\CA$--edge and a $\CB_0$--edge.
  Furthermore it is not the vertex of a chain as such a chain would
  lift to a segment in $\tilde{\Lambda}^H$ adjacent to this vertex and
  contained in both $K$ and $L$, contradicting the fact that their
  intersection is a single vertex.  Hence the image of $K \cap L$ is an
  essential vertex.

  Now suppose $K \cap L$ is a nondegenerate segment.  Its image in
  $\Lambda^H$ is clearly a maximal chain.  Furthermore, as $K$ is not
  contained in $L$, and $L$ is not contained in $K$, the chain is
  essential.
  
  For the converse, we show how to find the subtrees $K$ and $L$.  Let $\Lambda_A$ be the complement in $\Lambda^H$ of the union of the interiors of the $\CB_0$--edges.  There is exactly one component of $\Lambda_A$ that
  contains the given maximal essential chain or essential vertex.  Let
  $K$ be a lift of this component to $\tilde{\Lambda}^H$, and notice
  that $K \in \CV^H(\CA)$.  Similarly, let $\Lambda_{B_0}$ be the
  complement in $\Lambda^H$ of the union of the interior of the
  $\CA$--edges.  Attach each chain in $\alpha(\Lambda^H)$ to
  $\Lambda_{B_0}$ along its vertices to the appropriate component and
  call the resulting set of components $\Lambda_B$.  Again, there is
  exactly one component of $\Lambda_B$ that contains the given maximal
  essential chain or vertex.  Let $L$ be a lift of this component to
  $\tilde{\Lambda}^H$ that intersects $K$, and notice that $L \in
  \CV^H(\CB)$.  The given maximal essential chain or essential vertex
  is the image of $K \cap L$.
\end{proof}

By construction, two edges $e_{K_1}^{L_1}$ and $e_{K_2}^{L_2}$ in
$T^H$ are identified by $h \in H$ if and only if $h^{\pm 1}(K_1 \cap
L_1) = K_2 \cap L_2$.  Hence edges of $T^H/H$ correspond to maximal
essential chains and essential vertices in $\Lambda^H$.  Furthermore,
as the action of $\tilde{\Lambda}^H$ is free, an edge $e_K^L$ has a
nontrivial edge stabilizer if and only if $K \cap L$ is an infinite
line, in which case the corresponding essential chain in $\Lambda^H$
has two vertices that are identified. We say that an essential chain
$\alpha$ in $\Lambda^H$ is {\it simply connected} if the elements of
$\CV(\alpha)$ are all distinct.  As edge stabilizers are maximal
cyclic subgroups, an edge in $E^H(\CA,\CB)$ has nontrivial stabilizer
if and only if it corresponds to a non-simply connected chain. The
subset of simply connected maximal essential chains is denoted
$\alpha_{sc}(\Lambda^H)$.


We have now proved: 
\begin{thm}\label{prop:amalgam-free-volume}
  Suppose that $T$ is a very small cyclic tree dual to a splitting
  $F_k = A *_{\la c \ra} \la c, B_0 \ra$ and $H$ is a finitely
  generated subgroup of $F_k$.  Let $\CA \cup \CB_0$ be a basis
  relative to $T$, and define $\Lambda = \Lambda_{\CA \cup \CB_0}$ and
  $\Lambda^H = \tilde{\Lambda}^H/H$.  Then:
  \[ \vol_T(H) = \#|\alpha_{sc}(\Lambda^H)|+
  \#|\CV_{ess}(\Lambda^H)|.\]\qed
\end{thm}

\begin{ex}\label{ex:amalgamated}
  Let $T$ be the cyclic tree dual to the splitting $F_3 = \la a,b
  \ra*_{[a,b]} \la [a,b],c \ra$.  Then the basis $\{ a ,b \} \cup \{ c
  \}$ is relative to this splitting.  Let $H$ be a subgroup in the
  conjugacy class represented by the graph in Figure
  \ref{fig:amalgamated}.  Chains are denoted by dotted lines, all of
  which are essential, and essential vertices are black. There are two simply connected chains and nine essential vertices; hence $\vol_T(H) = 11$.  In Figure \ref{fig:amalgamated2} we
  demonstrate the vertex groups of the induced graph of groups
  decomposition $T^H/H$.  The underlying graph of $T^H/H$ has three
  vertices: $v_1,v_2$ and $v_3$.  There are seven edges from $v_1$ to $v_2$
  and five edges from $v_2$ to $v_3$, one of which has a nontrivial
  stabilizer.
  \begin{figure}[t]
    \centering
    \includegraphics{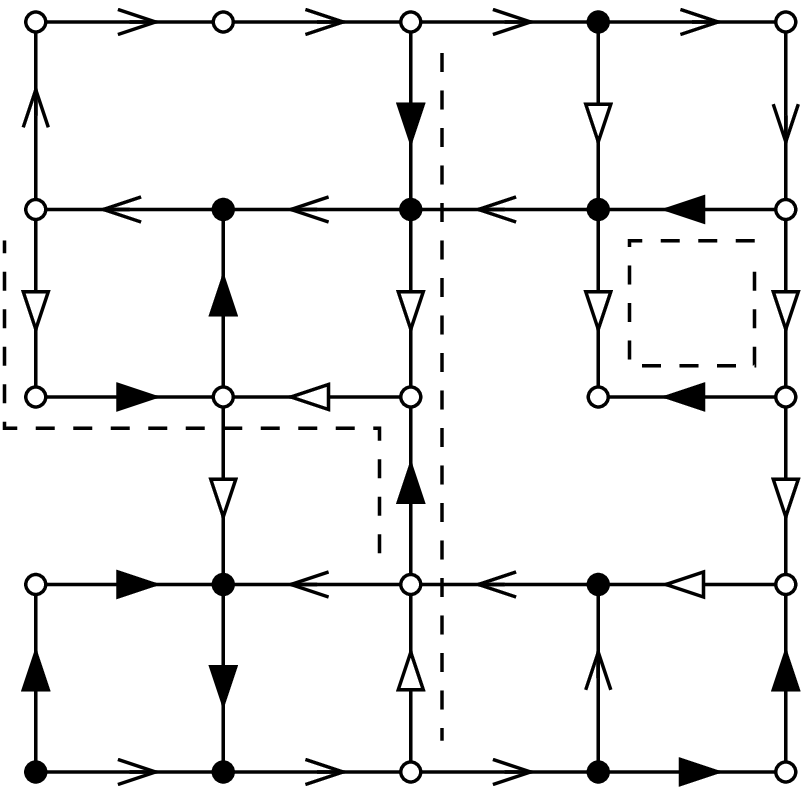}
    \caption{The graph $\Lambda^H$ in Example \ref{ex:amalgamated}.
      The arrows describe the immersion $\Lambda^H \to \Lambda$.  The
      black arrows are sent to the petal corresponding to ``$a$'', the
      white arrows to ``$b$'' and the open arrows to ``$c$''.}
    \label{fig:amalgamated}
  \end{figure}
  \begin{figure}[t]
    \centering
    \psfrag{a}{$v_1$}
    \psfrag{b}{$v_2$}
    \psfrag{c}{$v_3$}
    \includegraphics{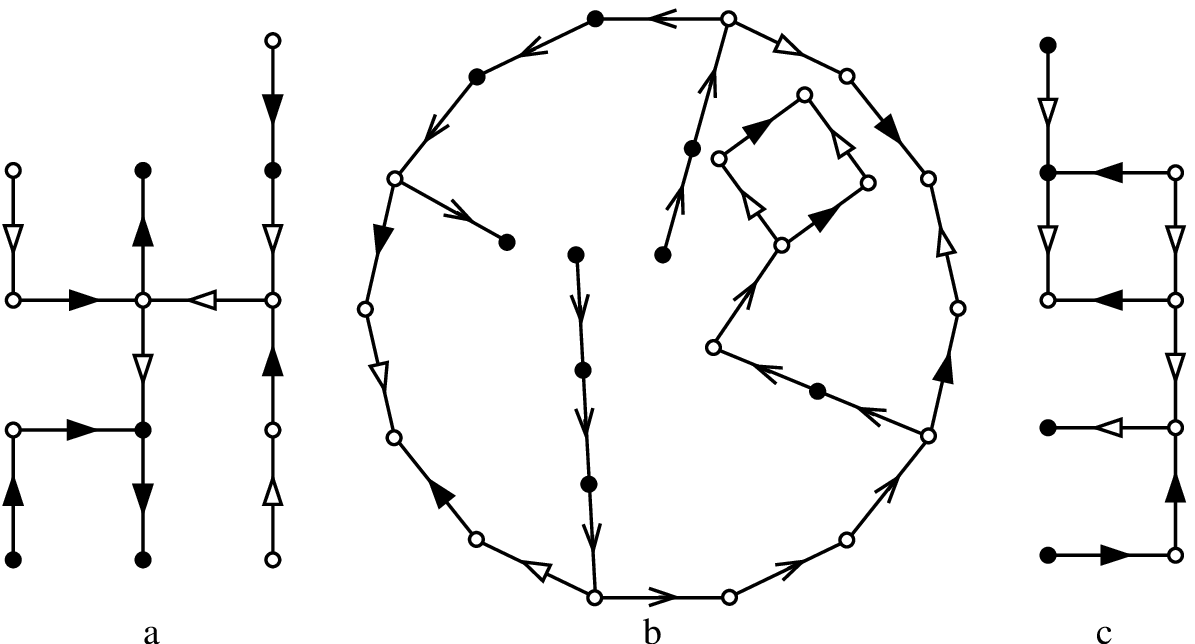}
    \caption{Graphs representing the conjugacy class of the vertex
      groups of the graph of groups decomposition $T^H/H$ in Example
      \ref{ex:amalgamated}.}
    \label{fig:amalgamated2}
  \end{figure}
\end{ex}

We state one final definition which will be used in Section
\ref{sc:growth}.


\begin{defi}\label{def:amal-cross}
  Let $H$ and $\Lambda^H$ be as in Theorem
  \ref{prop:amalgam-free-volume}.  A vertex of $\Lambda^H$ is a {\it
    crossing vertex} if it is either essential, or if it is a vertex
  of an essential chain and is adjacent to a $\CB_0$--edge.
\end{defi}


\subsection{Free volume for an HNN-extension}\label{ssc:HNN}

Now suppose that we have a cyclic HNN-extension
\[ F_k = (A_0 * \la t_0^{-1} c t_0 \ra) *_{\la c \ra} \] as in
Corollary \ref{HNN-basis}, with $c \in A_0$ indivisible and $T$ a
cyclic tree.  Let $\CA_0 = \{ a_1,\ldots,a_{k-1}\}$ be a basis for
$A_0$.  Then $\CA_0 \cup \{t_0\}$ is a basis for $F_k$ relative to
$T$.  Let $\Lambda_{\CA_0}$ be the $(k-1)$--petaled rose labeled by the
elements of $\CA_0$, and let $\Lambda = \Lambda_{\CA_0 \cup \{t_0\}}$
be the $k$--petaled rose labeled by the basis $\CA_0 \cup \{ t_0 \}$.
There is a natural inclusion $\iota_{\CA_0}: \Lambda_{\CA_0} \to
\Lambda$ which lifts to an immersion $\tilde{\iota}_{\CA_0}\co
\tilde{\Lambda}_{\CA_0} \to \Lambda$. Now let $\Lambda_\CA$ be the
$k$--rose, labeled by the elements of $\CA_0 \cup \{ t_0ct_0^{-1} \}$.
There is a natural map $\iota_\CA\co\Lambda_\CA \to \Lambda$ which
lifts to a map $\tilde{\iota}: \tilde{\Lambda}_\CA \to \Lambda$ from
the universal cover of $\Lambda_\CA$.  As before, we say that an edge
of $\Lambda$ corresponding to an element of $\CA_0$ is an
$\CA_0$--edge, and that an edge of $\Lambda$ corresponding to $t_0$ is
a $\{t_0\}$--edge.  A $\{t_0\}$--edge is positively oriented if it
corresponds to $t_0$ and negatively oriented if it corresponds to
$t_0^{-1}$.

Let $\CV(\CA)$ be the set of lifts of $\tilde{\iota}:
\tilde{\Lambda}_{\CA} \to \Lambda$ to $\tilde{\Lambda}$. Each lift
corresponds uniquely to a vertex of $T$, and two vertices are adjacent
if their $\ocap$--intersection of the two corresponding subtrees of
$\tilde{\Lambda}$ is nonempty and hence an infinite line. Let
$\CE(\CA)$ denote the set of all such pairwise $\ocap$--intersections
between elements of $\CV(\CA)$. Let $H$ be a finitely generated
subgroup of $F_k$, and let $\tilde{\Lambda}^H$ be its minimal subtree
in $\tilde{\Lambda}$. We denote by $\CV^H(\CA)$ the set consisting of
nontrivial subtrees of the form $K^H = \tilde{\Lambda}^H \ocap K$ for
$K \in \CV(\CA)$ which are not properly contained in a subtree
$\tilde{\Lambda}^H \ocap$ for any other $L \in \CV(\CA)$. We then let
$\CE^H(\CA)$ denote the set of (possibly trivial) subtrees $K^H \cap
L^H$ of trees $K^H$ and $L^H$ in $\CV^H(\CA)$.  Lemma
\ref{lm:minimaltree} transfers readily to the HNN-case, and so we have
a hold on the minimal subtree $T^H$.

A chain in $\Lambda^H$ is defined as in the amalgamated setting for
the conjugacy class of $c \in F_k$. We define vertices of a chain and simple connectivity of chain as before.

We refer to an edge in $\Lambda^H$ as an $\CA_0$--edge or $\{t_0\}$--edge 
according to its image in $\Lambda$.  A chain $\alpha$ is {\it
  nonessential} if:

\begin{enumerate}
\item any edge adjacent to $\alpha$ is a positively oriented
  $\{t_0\}$--edge which is adjacent to $\alpha$ at a vertex in $\CV(\alpha)$;
  or
\item $\alpha$ is only adjacent to $\CA_0$--edges and negatively
  oriented $\{t_0\}$--edges.
\end{enumerate}
Otherwise we say that $\alpha$ is {\it essential}.  As in the case of
amalgamated free products, the positively oriented $\{t_0\}$--edges
adjacent to a nonessential chain are considered $\CA_0$--edges.  The
set of all maximal essential chains on $\Lambda^H$ is denoted by
$\alpha(\Lambda_H)$.  The subset of simply connected essential chains
is denoted $\alpha_{sc}(\Lambda_H)$.

We say that a vertex is {\it essential} if it is the initial vertex of
a positively oriented $\{t_0\}$--edge, but is not a chain vertex of any
chain. The set of all essential vertices we denote by
$\CV_{ess}(\Lambda_H)$.

With these definitions in place, we give an analogue of Lemma
\ref{lm:chains} whose proof is similar.

\begin{lem}
  With the notation above, the image of a subtree in $\CE^H(\CA)$ in
  $\Lambda^H$ is either a maximal essential chain or an essential
  vertex. Conversely, every maximal essential chain or vertex is the
  image of some subtree in $\CE^H(\CA)$.
\end{lem}

We can now state how to count free volume for a finitely generated
subgroup with respect to a cyclic tree dual to an HNN-extension, as
the argument now proceeds as for the amalgamation case.


\begin{thm}\label{prop:hnn-free-volume}
  Suppose that $T$ is a very small cyclic tree dual to a splitting
  $F_k = (A_0 * \la t_0 c t_0^{-1} \ra) *_{\la c \ra}$ and $H$ is a
  finitely generated subgroup of $F_k$.  Let $\CA_0 \cup \{t_0\}$ be a
  basis relative to $T$, and define $\Lambda = \Lambda_{\CA_0 \cup \{
    t_0 \}}$ and $\Lambda^H = \tilde{\Lambda}^H/H$.  Then:
  \[\vol_T(H) = \#|\alpha_{sc}(\Lambda^H) | + \#
  |\CV_{ess}(\Lambda^H)|.\]\qed
\end{thm}


\begin{ex}\label{ex:HNN}
  Here we let $T$ be the cyclic tree dual to the splitting $F_3 =
  \langle a, b, t_0^{-1} [a,b] t_0 \rangle \ast_{\langle [a,b]
    \rangle}$, with cyclic edge generator $c = [a,b]$. Let $H$ be a
  subgroup in the conjugacy class represented by the graph in Figure
  \ref{fig:HNN}. The eight chains are indicated by dotted lines; three
  of these are inessential, and one is not simply connected. There is
  a single essential vertex, indicated in black. The free volume is
  therefore $\vol_T(H) = 5$.
\end{ex}

\begin{figure}[t]
  \centering
  \includegraphics{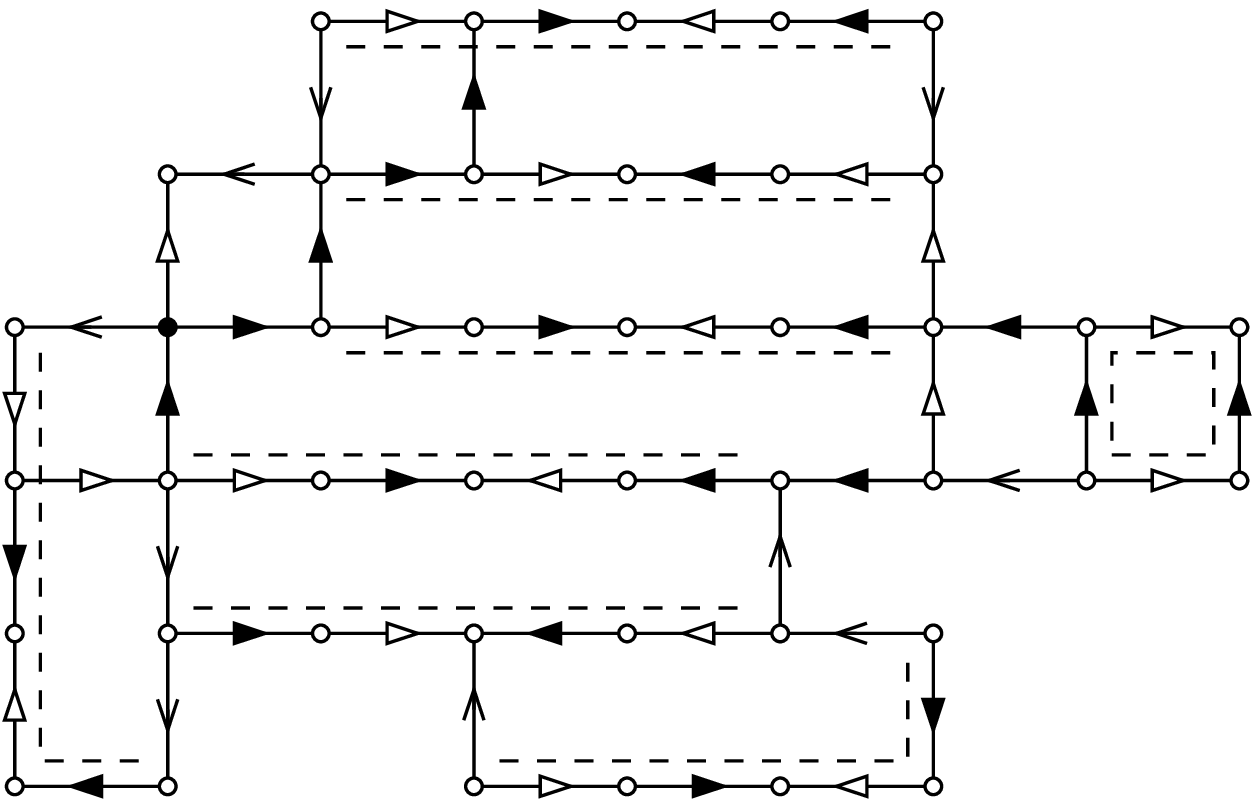}
  \caption{The graph $\Lambda^H$ in Example \ref{ex:HNN}. The arrows
    describe the immersion $\Lambda^H \to \Lambda$. The black arrows
    are sent to the petal corresponding to ``$a$'', white arrows to
    ``$b$'', and the open arrows to ``$t_0$''. Chains are indicated by
    dotted line segments.}\label{fig:HNN}
\end{figure}

\begin{figure}[t]
  \centering
  \includegraphics{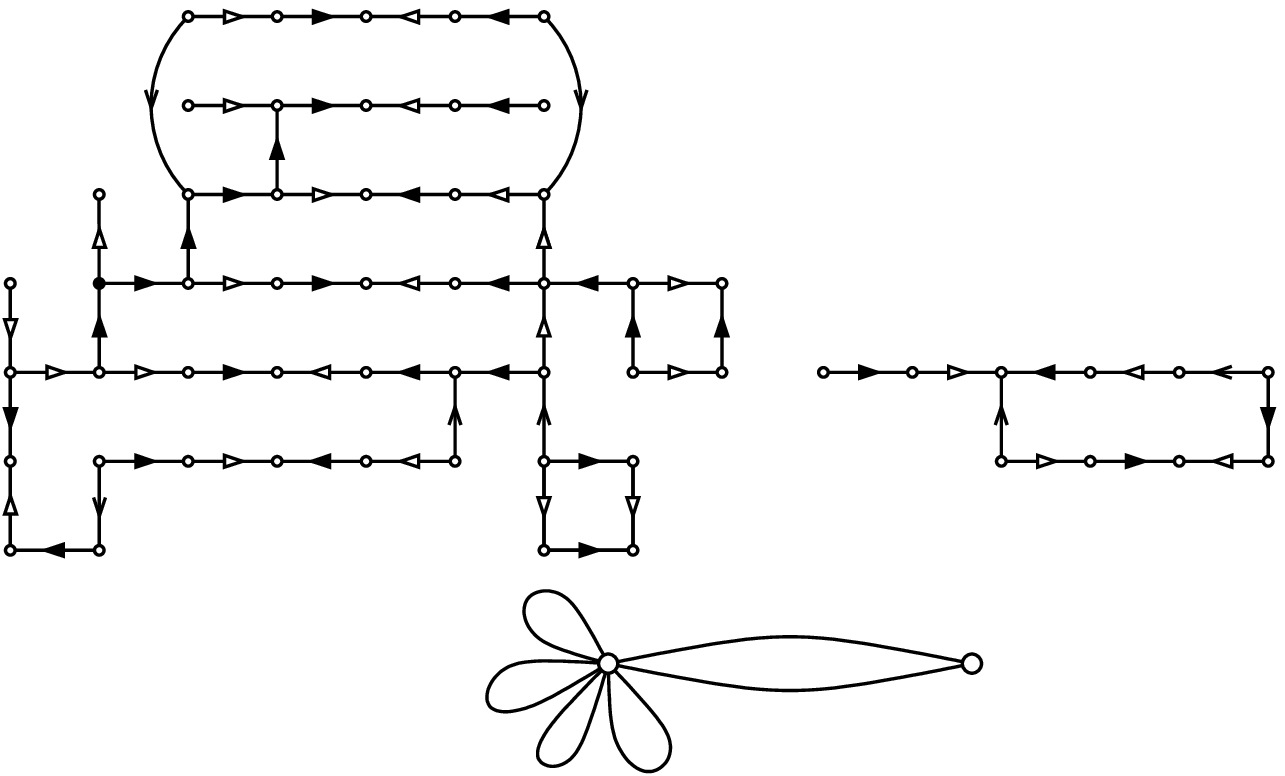}
  \caption{ The top two graphs represent the conjugacy class of the
    vertex groups of the graph of groups decomposition $T^H/H$ in
    Example \ref{ex:HNN}. The graph below represents the graph of
    groups $T^H/H$.}
\end{figure}


Again we have a notion of crossing vertex for an HNN-extension similar
to Definition \ref{def:amal-cross}.

\begin{defi}\label{def:hnn-cross}
  Let $H$ and $\Lambda^H$ be as in Theorem
  \ref{prop:hnn-free-volume}.  A vertex of $\Lambda^H$ is a {\it
    crossing vertex} if it is an essential vertex, or if it is a vertex of
  an essential chain and is adjacent to a positively oriented $\{t_0\}$--edge.
\end{defi}


\section{Twisted volume growth}\label{sc:growth}

Let $T_1$ and $T_2$ be two very small cyclic trees for $F_k$ with edge
stabilizers respectively generated by conjugates of the elements $c_1$
and $c_2$ and with associated Dehn twist elements $\delta_1$ and
$\delta_2$.  Fix bases $\CT_1 = \CA_1 \cup \CB_1$ and $\CT_2 = \CA_2
\cup \CB_2$ for $F_k$ relative to these trees. 
Let $\Lambda_1 = \Lambda_{\CT_1}$ and $\Lambda_2 = \Lambda_{\CT_2}$ be
the $k$--petaled roses for these bases, as constructed in Section
\ref{counting}.  
  
The goal of this section is to prove Theorem \ref{th:B} of the
introduction; that is, we want to find bounds for
$\vol_{T_2}(\delta^{\pm n}_1(H))$ when $H$ is a finitely generated
malnormal or cyclic subgroup. To begin, we discuss how the graph of
groups decomposition described in Section \ref{counting} of a finitely
generated subgroup $H$ and the according free volume of $H$ changes
upon twisting.

\subsection{Graph composition}\label{ssc:composition}

Let $\nu \co \Lambda_1 \to \Lambda_2$ be a (linear) homotopy
equivalence representing the change in marking. Suppose $\rho\co \CH
\to \Lambda_1$ is a map (not necessarily an immersion) such that the
image of $\pi_1(\CH)$ in $\pi_1(\Lambda_1)$ is a conjugate of $H$.
Then we can form the composition $\nu \circ \rho \co \CH \to
\Lambda_2$.  We define $\CH_{\Lambda_2}$ as the graph (equipped with
the map $\rho_{\Lambda_2} \co \CH_{\Lambda_2} \to \Lambda_2$) obtained
from $\CH$ by subdividing each edge $e \subset \CH$ so that every component of the pre-image of the vertex in $\Lambda_2$ is a vertex.  We say that 
$\CH_{\Lambda_2}$ is obtained from $\CH$ by {\it graph composition
  using $\nu$.}

The following lemma is clear from the definitions.

\begin{lem}\label{composition}
  After folding and pruning the map $\rho_{\Lambda_2}\co
  \CH_{\Lambda_2} \to \Lambda_2$ we obtain an immersion $\rho_2^H\co
  \CG_2^H \to \Lambda_2$ of a core graph $\CG_2^H$ for the subgroup
  $H$.\qed
\end{lem}

\subsection{Graph surgery}\label{ssc:surgery}

Fix an immersion of a core graph $\rho_1^H\co \CG_1^H \to
\Lambda_1$. We label edges as $\CA_1$--edges or $\CB_1$--edges according
to their image in $\Lambda_1$, as in Section \ref{counting}. We then
locate the simply connected and non-simply connected chains, and the
essential and nonessential vertices and chains. Recall that a {\it
  crossing vertex} is described in Definitions \ref{def:amal-cross}
and \ref{def:hnn-cross}

For $n \geq 0$, let $a_n = [0,1]$ be an interval subdivided into
$|c_1^n|_{\CT_1}$ edges and let $\bar{a}_n$ denote $a_n$ with
opposite orientation.  Let $v \in \CG_1^H$ be a crossing vertex.  Add
a new vertex $v'$ and insert a copy of the the interval $a_n$ by
attaching the vertex $0$ to $v$ and the vertex $1$ to $v'$.  Now
perform one of the two following operations:
\begin{enumerate}

\item If $T_1$ is dual to an amalgamated free product, then for each
  $\CB_1$--edge $e$ adjacent to $v$, redefine the initial vertex of
  $e$ to be $v'$.

\item If $T_1$ is dual to an HNN-extension (so that $\CB_1$ is equal
  to the one-element set $\{t_0\}$ for some $t_0$), then redefine to be $v'$ the initial vertex of the unique positively oriented $\{t_0\}$--edge adjacent to $v$.

\end{enumerate}

Let $\Upsilon_n^H$ be the graph obtained by performing the above
appropriate operation at each crossing vertex of $\CG_1^H$.  Define a
map $\rho_1\co \Upsilon_n^H \to \Lambda_1$ which is equal to $\rho^H_1$
on edges of $\CG_1^H$, and which maps each new arc $a_n$ to the edge
path for $c_1^n$ in $\Lambda_1$.  We say that $\Upsilon_n^H$ is obtained
from $\CG^H_1$ by {\it graph surgery along $T_1$}.

\begin{lem}\label{core1}
  After folding and pruning the map $\rho_1\co \Upsilon_n^H \to
  \Lambda_1$, we obtain the immersion of the core graph
  $\rho_1^{\delta_1^n(H)}\co \CG_1^{\delta_1^n(H)} \to \Lambda_1$ for
  the subgroup $\delta_1^n(H)$.
\end{lem}

\begin{proof}
  Let $\nu$ and $\CH$ be as in Section \ref{ssc:composition}, where
  $\Lambda_2$ is the $k$--petaled rose corresponding to the image of
  the basis $\CA_1 \cup \CB_1$ under the Dehn twist $\delta_1$. Recall
  that this means that the petals of $\Lambda_2 $ correspond to
  elements of the basis $\CA_1 \cup c_1\CB_1c_1^{-1}$ if $T_1$ is dual to
  an amalgamated free product, and to the basis $\CA_1 \cup c_1\CB_1 $
  if $T_1$ is dual to an HNN-extension. Also as in Section
  \ref{ssc:composition}, let $\CH_{\Lambda_2}$ be the graph obtained
  from $\CH$ by graph composition using $\nu$. If $T_1$ is dual to an
  HNN-extension, then $\CH_{\Lambda_2}$ is equal to
  $\Upsilon_n^H$. Otherwise the graph $\Upsilon_n^H$ is obtained from
  $\CH_{\Lambda_2}$ by folding and pruning segments corresponding to
  $\bar{a}_na_n$ between adjacent $\CB_1$--edges.
\end{proof}

It is clear that by inserting $\bar{a}_n$ at each crossing vertex to
obtain $\Upsilon_{-n}^H$, we can fold and prune to obtain an
immersion of a core graph $\CG_1^{\delta_1^{-n}(H)}$ for the subgroup
$\delta_1^{-n}(H)$.

Notice that if the crossing vertex $v$ lies on a non-simply connected
chain, then the entire newly added interval $a_n$ can be folded onto
this chain; it is for this reason that we record free volume instead of total
volume.  Combining Lemmas \ref{composition} and \ref{core1} we obtain
the following corollary describing the change in the graph of groups
decomposition for $H$ upon twisting.

\begin{cor}\label{co:surgery}
  Suppose $\rho_1^H\co \CG_1^H \to \Lambda_1$ is an immersion of a
  core graph for $H$ and let $\rho_1\co \Upsilon_N^H \to \Lambda_1$ be
  the result obtained by graph surgery along $T_1$. Then after folding and
  pruning the composition $\nu \circ \rho_1 \co \Upsilon_n^H \to
  \Lambda_2$, we obtain an immersion $\rho_2^{\delta_1^n(H)}\co
  \CG_2^{\delta_1^n(H)} \to \Lambda_2$ of a core graph
  $\CG_2^{\delta_1^n(H)}$ for the subgroup $\delta^n_1(H)$.  \qed
\end{cor}

In the next section we show how to control the amount of folding and
pruning that takes place on the newly added intervals $a_n$ in the
above corollary.

\subsection{Safe essential pieces}\label{ssc:safe}

Suppose that $T_2$ is a very small cyclic tree dual to an amalgamated
free product.  By conjugating the basis $\CT_1$ (so that it remains a
basis relative to $T_1$ and so that the associated Dehn twist automorphism
defines the same outer automorphism class), we can assume that $c_1$
is cyclically reduced with respect to $\CT_2$.  Moreover, if $c_1$
does not fix a point in $T_2$, then by further conjugating we can
assume that, as a reduced word in $\CT_2$, the element $c_1$ has the
form:
\begin{equation}\label{eq:reduced-amal}
  c_1 = x_1c_2^{i_1}y_1c_2^{j_1} \cdots x_mc_2^{i_m}y_mc_2^{j_m}
\end{equation}
where for $r =1,\ldots,m$, the word $y_r$ is a nontrivial word in
$\CB_2$ and the word $x_r$ is a nontrivial word in $\CA_2$ such that
$zx_r$ and $x_rz$ are reduced for $z = c_2,c_2^{-1}$.  (This last
statement requires the adjective very small.)  Thus $|c_1^n |_{\CT_2}
= n|c_1|_{\CT_2}$ and $\ell_{T_2}(c_1^n) = 2mn$.

Now suppose that $T_2$ is a very small cyclic tree for an
HNN-extension.  Again by conjugating the basis $\CT_1$, we can assume
that $c_1$ is cyclically reduced with respect to $\CT_2$.  Moreover,
if $c_1$ does not fix a point in $T_2$, then by further conjugating,
we can assume that as a reduced word in $\CT_2$, the element $c_1$ has
the form:
\begin{equation*}\label{eq:reduced-HNN}
  c_1 = x_1(c_2^{i_1}t_0)^{\epsilon_1}x_2(c_2^{i_2}t_0)^{\epsilon_2} \cdots 
  x_m(c_2^{i_m}t_0)^{\epsilon_m}  
\end{equation*}
where for $r = 1, \ldots,m$, the word $x_r$ is a (possibly trivial)
word in $\CA_2 \cup \{t_0^{-1}c_2t_0 \}$, where $\epsilon_r \in \{ \pm
1\}$; and if $\epsilon_r = 1$, then $x_rz$ is a reduced word for $z =
c_2,c_2^{-1}$, and if $\epsilon_r = -1$, then $zx_{r+1}$ is a reduced
word for $z = c_2,c_2^{-1}$, where the subscript is considered modulo
$m$.  (Again, this last statement requires the adjective very small.)
Thus $|c_1^n|_{\CT_2} = n|c_1|_{\CT_2}$ and $\ell_{T_2}(c_1^n) = mn$.

In either of two above cases, we say that $c_1$ is
\emph{$T_2$--reduced}.  For the remainder of this section, we will
always assume that $c_1$ is $T_2$--reduced.

Let $\alpha^n_{\Lambda_2} = [0,1]$ be the interval subdivided into
$|c_1^n|_{\CT_2}$ edges.  There is a map $\alpha^n_{\Lambda_2} \to
\CG_2^{\langle c_1^n \rangle} \to \Lambda_2$, where the first map
identifies the endpoints of $\alpha^n_{\Lambda_2}$, and the second map
is the immersion of the core graph whose image represents the
conjugacy class of $c_1^n$.  As $c_1$ is cyclically reduced with
respect to $\CT_2$, no folding takes place after identifying the
vertices of $\alpha^n_{\Lambda_2}$.  Also, as $c_1$ is $T_2$--reduced,
essential chains and essential vertices relative to the basis $\CT_2$
can be considered as subsets of $\alpha^n_{\Lambda_2}$.  These
essential chains and essential vertices are referred to as {\it
  essential pieces} (relative to $\CT_2$).

We say that an essential piece in $\alpha^n_{\Lambda_2}$ is {\it safe}
if the vertex or chain does not intersect a vertex of one of the
extremal $BCC(\CT_1, \CT_2)$ edges of $\alpha^n_{\Lambda_2}$.  It is
clear that at most $2BCC(\CT_1, \CT_2) + 2$ essential pieces in
$\alpha^n_{\Lambda_2}$ are not safe. 

\begin{ex}\label{ex:safe}
  Let $T_2$ be the cyclic tree dual to the splitting $F_3 = \la a,c
  \ra *_{\la c \ra} \la c,b \ra$.  Suppose $T_1$ is another cyclic
  tree such that $c_1 = ababac^3b$ (this is $T_2$--reduced) and
  $BCC(\CT_1,\CT_2) = 3$.  The segment $\alpha^1_{\Lambda_2}$ is shown
  in Figure \ref{fig:chain2}.  The only safe essential piece is the
  fifth from the left essential vertex.

\begin{figure}[t]
  \psfrag{s}{safe} 
  \centering
  \includegraphics{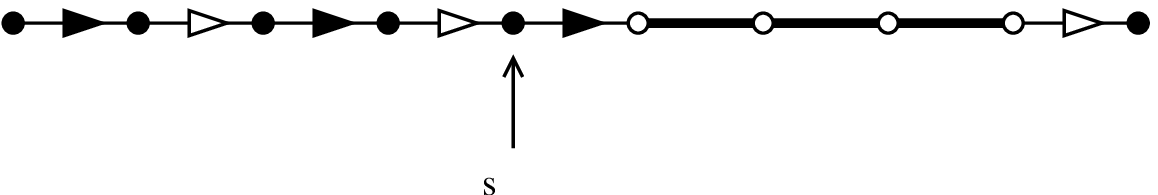}
  \caption{The segment $\alpha^1_{\Lambda_2}$ for $T_2$ in Example
    \ref{ex:safe}.  The black arrows are sent to the petal
    corresponding to ``$a$'', white arrows to ``$b$'' and the thick
    line without arrows represents an essential chain.  Essential
    vertices are black.}
  \label{fig:chain2}
\end{figure}
\end{ex}

Consider an immersion of a core graph $\rho\co \CG^H_1 \to \Lambda_1$.
The image of a chain $\alpha = (\gamma_1,\ldots,\gamma_\ell) \in
\alpha(\CG^H_1)$ in $(\CG_1^H)_{\Lambda_2}$, the graph composition of
$\CG^H_1$ using $\nu \co \Lambda_1 \to \Lambda_2$, is naturally
identified with a copy of the segment $\alpha^\ell_{\Lambda_2}$.

To obtain the inequality of Theorem \ref{th:B}, we determine the
number safe pieces resulting from twisting which contribute to new
volume.  Upon twisting, safe essential pieces might get folded with
surgered segments and then pruned. We account for these pruned safe
pieces by showing that they must contribute to the original free
volume of $H$ with respect to $T_2$. This is the use of the
following proposition.

\begin{prop}\label{prop:safe}
  Suppose that $H$ is a finitely generated malnormal or cyclic
  subgroup of $F_k$ where $\rank(H) \leq R$.  Then there is an $M =
  M(R)$, such that given $\rho\co \CG_1^H \to \Lambda_1$, an immersion
  of the core graph $\CG^H_1$, we have:
  \begin{eqnarray}\label{eq:safe}
    \sum_{(\gamma_1, \ldots, \gamma_\ell) \in \alpha(\CG^H_1)} \#\mbox{safe essential pieces in } \alpha^\ell_{\Lambda_2} \leq M(\vol_{T_2}(H) + 1).
  \end{eqnarray}
\end{prop}

\begin{proof} 

  Recall that $\CG_2^H$ is the core graph obtained from folding and
  pruning the graph $\rho_{\Lambda_2}\co (\CG_1^H)_{\Lambda_2} \to
  \Lambda_2$ obtained in Lemma \ref{composition}. As the subgroup $H$
  is malnormal or cyclic, it follows from Lemma
  \ref{lm:malnormal-freevolume} that the number of edges in the graph
  $T_2^H/H$ exceeds $\vol_{T_2}(H)$ by at most one.  Therefore from
  the discussions in Sections \ref{ssc:amalgamated} and \ref{ssc:HNN},
  it follows that there is at most one essential chain of $\CG_2^H$
  relative to $\CT_2$ which is not simply connected and hence
  does not contribute to $\vol_{T_2}(H)$. We will thus demonstrate
  \eqref{eq:safe} by showing that the sum on the left hand side of the
  inequality is less than $M$ times the number of essential chains and
  vertices in $\CG^H_2$.

  Let $\alpha = (\gamma_1,\ldots,\gamma_\ell)$ be a chain in
  $\CG^H_1$.  As $c_1$ is $T_2$--reduced and by bounded cancellation,
  any safe essential piece in $\alpha^\ell_{\Lambda_2}$ survives as a
  subset after folding and pruning $(\CG_1^H)_{\Lambda_2}$ to get
  $\CG_2^H$.  What needs to be shown is that any such safe
  essential piece of a chain in $\CG_1^H$ is part of an essential
  piece of $\CG^H_2$ and that over all chains in $\alpha(\CG^H_1)$,
  only boundedly many safe pieces are combined into the same essential
  vertex or chain.

  If $|c_1^\ell|_{\CT_2} \leq 2BCC(\CT_1,\CT_2)+2$ then there are no safe
  pieces in $\alpha^\ell_{\Lambda_2}$.  Otherwise, decompose the segment
  $\alpha^\ell_{\Lambda_2}$ as $xe_1ye_2z$, where 
  \[ |x|_{\CT_2} = |z|_{\CT_2} = BCC(\CT_1,\CT_2)\] and $e_1$ and
  $e_2$ are single edges.  Thus any safe essential pieces of
  $\alpha^\ell_{\Lambda_2}$ is contained in $y$ and the segment
  $e_1ye_2$ survives folding (although some of its vertices and edges
  may be identified).
 
  Consider an essential vertex $v$ in
  $\alpha^\ell_{\Lambda_2}$.  Thus $v$ is adjacent to an $\CA_2$--edge 
  of $e_1ye_2$ not labeled $c_2$, as well as a $\CB_2$--edge (a positively
  oriented $\{t_0\}$--edge in the case when $T_1$ is dual to an
  HNN-extension) of $e_1ye_2$.  Hence, as these edges remain after
  folding and pruning, $v$ is an essential vertex in $\CG^H_2$ unless
  it is part of a chain.  Such a chain could not cross either of the
  edges of $e_1ye_2$ that are adjacent to $v$.  Thus such a chain is
  necessarily essential due to the edges in $e_1ye_2$ adjacent
  to $v$.  Similarly, an essential chain in $\alpha^\ell_{\Lambda_2}$ is
  part of an essential chain in $\CG^H_2$ (it may not be maximal in
  $\CG^H_2$).

  If $\rank(H)=1$, then $(\CG^H_1)_{\Lambda_2}$ is a circle and as such the
  segment $e_1ye_2$ is embedded in $\CG^H_2$ and essential vertices
  and chains of $e_1ye_2$ are not contained in a larger essential
  chain of $\CG^H_2$.  Hence, for $M = 1$, inequality \eqref{eq:safe}
  holds.

  Now suppose that $R \geq \rank(H) > 1$ and $v$ and $v'$ are vertices
  of $(\CG^H_1)_{\Lambda_2}$ that are identified in $\CG^H_2$, where
  $v$ is contained in an essential safe piece arising from $\alpha \in
  \alpha(\CG^H_1)$.  There is an edge path $\beta$ in
  $(\CG^H_1)_{\Lambda_2}$ connecting $v$ to $v'$ which is folded.
  Notice, the number of edges of $\beta$ is bounded by
  $2BCC(\CT_1,\CT_2)$.  As $v$ is not in the extremal
  $BCC(\CT_1,\CT_2)$ edges of $\alpha^\ell_{\Lambda_2}$ the path
  $\beta$ does not contain a component of $\alpha^\ell_{\Lambda_2} -
  \{ v \}$ and therefore intersects a vertex of valence at least three
  in $(\CG^H_1)_{\Lambda_2}$.  For any $R$, there are boundedly many
  paths in a graph of rank at most $R$ with at most
  $2BCC(\CT_1,\CT_2)$ edges that contain a vertex of valence at least
  3.  Thus over all chains $\alpha \in \alpha(\CG_1^H)$, only
  boundedly many safe essential pieces of $\alpha^\ell_{\Lambda_2}$
  are contained in a given essential vertex of chain of $\CG_2^H$ after
  folding.  Taking $M$ to be this bound, inequality \eqref{eq:safe}
  holds.
\end{proof}

\subsection{Linear growth}\label{ssc:linear}

We can now prove our theorem giving a linear lower bound on the free
volume of a finitely generated malnormal or cyclic subgroup after
iterating a Dehn twist.  Although it is not needed here, we also prove the (easier) linear upper bound; this upper bound in applied in \cite{Clay-Pettet}.


\begin{thm}\label{th:B}
  Let $\delta_1$ be a Dehn twist associated to the very small
  cyclic tree $T_1$ with edge stabilizers generated by conjugates of
  the element $c_1$ and let $T_2$ be any other very small cyclic tree.
  Then there exists a constant $C = C(T_1,T_2)$ such that for any
  finitely generated malnormal or cyclic subgroup $H \subseteq F_k$
  with $\rank(H) \leq R$ and $n \geq 0$ the following hold:
  \begin{align}
  \vol_{T_2}(\delta_1^{\pm n}(H)) &\geq 
  \vol_{T_1}(H) \bigl( n\ell_{T_2}(c_1) - C \bigr) - 
  M\vol_{T_2}(H)\label{eq:lowerbound} \\
  \vol_{T_2}(\delta_1^{\pm n}(H)) &\leq 
  \vol_{T_1}(H) \bigl( n\ell_{T_2}(c_1) + C \bigr) + 
  M\vol_{T_2}(H)\label{eq:upperbound}
\end{align}
where $M$ is the constant from Proposition \ref{prop:safe}.
\end{thm}


\begin{proof}
  We will only show this for $\delta_1^n$; it will then be clear how to
  modify the argument for $\delta_1^{-n}$.

  Recall that $\CT_1 = \CA_1 \cup \CB_1$ and $\CT_2 = \CA_2 \cup
  \CB_2$ are bases for $F_k$ relative to the trees $T_1$ and $T_2$
  respectively, that $\nu\co \Lambda_1 \to \Lambda_2$ is a homotopy
  equivalence representing the change in marking and $\Lambda_1$ and
  $\Lambda_2$ are the $k$--petaled roses marked by $\CT_1$ and $ \CT_2$
  respectively.  Let $B = BCC(\CT_1, \CT_2)$ denote the bounded
  cancellation constant with respect to these bases.  Finally, let
  $\rho\co \CG^H_1 \to \Lambda_1$ be an immersion of a core graph for
  $H$.  We will first prove \eqref{eq:lowerbound}.

  If $\ell_{T_2}(c_1) = 0$ there is nothing to prove.  Otherwise,
  after replacing $\CT_1$ by a conjugate (replacing $\Lambda_1$ and
  $B$ accordingly) we can assume that $c_1$ is $T_2$--reduced.  We can
  assume that $C$ is large enough so that if
  $n\ell_{T_2}(c_1) \geq C$ then the segment $\alpha^{n-1}_{\Lambda_2}$
  contains a safe essential chain or vertex.  Notice that the number
  of safe essential pieces in $\alpha^n_{\Lambda_2}$ is at least
  $n\ell_{T_2}(c_1) - (2B + 2)$.

  Let $\Upsilon_n^H$ be the graph obtained from graph surgery on the
  core graph $\CG^H_1$ along $T_1$ equipped with the map $\rho_1\co
  \Upsilon_n^H \to \Lambda_1$.  Notice that at least $\vol_{T_1}(H)$
  segments $a_n$ have been added to $\CG_1^H$, as every essential
  piece contains at least one crossing vertex.  Further notice that
  since $c_1$ is indivisible and cyclically reduced with respect to
  $\CT_1$, the map $\rho_1\co \Upsilon_n^H \to \Lambda_1$ is an
  immersion, except possibly at an initial vertex of one of the
  surgered segments $a_n$.

  By Corollary \ref{co:surgery} the map $(\nu \circ
  \rho_1)_{\Lambda_2}\co (\Upsilon_n^H)_{\Lambda_2} \to \Lambda_2$
  folds to an immersion, which by pruning results in the immersion of
  the core graph $\rho_2^{\delta_1^n(H)}\co \CG_2^{\delta_1^n(H)} \to
  \Lambda_2$.  The image of each of the surgered segments $a_n$ in
  $(\Upsilon_n^H)_{\Lambda_2}$ is a copy of $\alpha^n_{\Lambda_2}$.
  We need to bound the number of essential chains and vertices
  belonging to copies of the segment $\alpha^n_{\Lambda_2}$ in
  $\Upsilon^H_{\Lambda_2}$ which get pruned. As the order in which
  folding occurs to arrive at $\CG^H_2$ does not matter, we will focus
  on a single surgered segment $a_n$ and its associated copy of
  $\alpha^n_{\Lambda_2}$ in $\Upsilon^H_{\Lambda_2}$.  Therefore, we
  assume that the only places where the map
  $(\Upsilon_n^H)_{\Lambda_2} \to \Lambda_2$ is not an immersion are
  the initial and terminal vertices of this copy of
  $\alpha^n_{\Lambda_2}$.

  If $a_n$ is surgered in at an essential vertex, then $\rho_1\co
  \Upsilon_n^H \to \Lambda_1$ folds at most the first $|c_1|_{\CT_1}$
  edges of $a_n$, as any additional folding would imply the presence
  of an essential maximal chain adjacent to the essential vertex (this
  uses the fact that $c_1$ is indivisible in $F_k$).  In particular,
  the terminal subsegment $a_{n-1}$ survives.  Hence after graph
  composition using $\nu$, at most the extremal $B$ edges of the
  corresponding copy of $\alpha^{n-1}_{\Lambda_2}$ are pruned.  As no
  other edges of $\CG^H_2$ intersect the remaining segment of
  $\alpha^{n-1}_{\Lambda_2}$ all safe pieces of $\alpha^n_{\Lambda_2}$
  are safe pieces of $\CG^{\delta_1^n(H)}_2$.

  Now suppose the crossing vertex is not an essential vertex. Hence
  there is an essential chain $\alpha = (\gamma_1,\ldots,\gamma_m)$
  such that either the crossing vertex is $\gamma_i(0)$ for some $i$, or the
  crossing vertex is $\gamma_m(1)$.  Without loss of generality (as we
  are only measuring the contribution of this essential chain), we can
  assume that the crossing vertex is rightmost along the chain.  If it
  is $\gamma_m(1)$, then as in the proceeding paragraph the map
  $\rho_1\co \Upsilon_n^H \to \Lambda_1$ folds at most the first
  $|c_1|_{\CT_1}$ edges of $a_n$ and all safe pieces of the remaining
  $\alpha^{n-1}_{\Lambda_2}$ are essential vertices or chains of
  $\CG^{\delta_1^n(H)}_2$.

  Otherwise the crossing vertex is $\gamma_i(0)$ for some $i$.  Then
  $\Upsilon_n^H \to \Lambda_1$ will fold more the the initial
  $|c_1|_{\CT_1}$ at the initial vertex of $a_n$.  Here we claim that
  at most $2B+2$ safe pieces of $\alpha^n_{\Lambda_2}$ that are folded
  and pruned are not first identified with a safe essential piece of
  $\alpha$.

  If $n \leq m - i$ then in $\Upsilon_n^H$, the entire segment
  $\alpha^n_{\Lambda_2}$ can be folded onto $\alpha^m_{\Lambda_2}$,
  identifying safe pieces of $\alpha^n_{\Lambda_2}$ with safe pieces
  of $\alpha^m_{\Lambda_2}$; such pieces may then be pruned in forming
  $\CG^{\delta_1^n(H)}_2$.  If $n > m - i$, then the terminal
  $\alpha^{m - i}_{\Lambda_2}$ segment of $\alpha^m_{\Lambda_2}$ can
  be folded onto $\alpha^n_{\Lambda_2}$.  When folding, safe pieces in
  an initial segment of $\alpha^n_{\Lambda_2}$ are identified with
  safe pieces of $\alpha^m_{\Lambda_2}$.  However some safe pieces of
  $\alpha^n_{\Lambda_2}$ are identified with non-safe pieces of
  $\alpha^m_{\Lambda_2}$ coming from essential pieces of
  $\alpha^m_{\Lambda_2}$ intersecting along the terminal $B+1$ edges of
  $\alpha^m_{\Lambda_2}$.  Thus the number of such safe pieces of
  $\alpha^n_{\Lambda_2}$ identified with non-safe pieces of
  $\alpha^m_{\Lambda_2}$ is bounded by $B+1$.  There may need to be
  additional folding at the terminal vertex of $\alpha^m_{\Lambda_2}$; however the amount of folding is bounded.  Indeed as $\alpha$ is
  maximal, at the terminal vertex $\alpha$ in $\CG^H_1$, we have to
  fold at most an additional $|c_1|_{\CT_1}$ edges.  Thus after
  folding the initial portion of $a_n$ over $\alpha$ and possibly at
  most $|c_1|_{\CT_1}$ edges, the induced map is an immersion at this
  vertex and hence at most $B$ of the initial edges in the terminal
  $\alpha^{n-m-1}_{\Lambda_2}$ segment of $\alpha^n_{\Lambda_2}$ are
  folded with other edges adjacent to this vertex. Therefore at most
  an additional $B$ edges are pruned, eliminating at most an
  additional $B+1$ safe pieces from $\alpha^n_{\Lambda_2}$.  This
  proves our claim.  See Figure \ref{fig:linear}.

  \begin{figure}[t]
    \centering
    \psfrag{f}{fold}
    \psfrag{a}{$\alpha^n_{\Lambda_2}$}
    \psfrag{b}{$\alpha^m_{\Lambda_2}$}
    \includegraphics{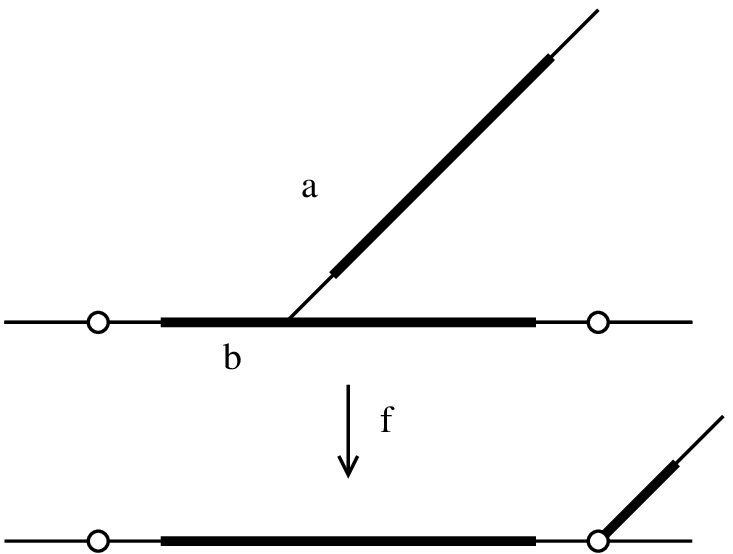}
    \caption{Folding the initial part of the surgered segments
      $\alpha^n_{\Lambda_2}$ to $\Upsilon^H_{\Lambda_2}$.  The safe
      pieces are contained in the thickened edges.  At most $B$ more
      edges of $\alpha^n_{\Lambda_2}$ need to be folded after this
      initial fold.}
    \label{fig:linear}
  \end{figure}

  Putting this claim together with Proposition \ref{prop:safe} and
  summing up over all crossing vertices of $\CG^H_1$, we see that the
  number of essential pieces of $\CG^{\delta^n_1(H)}_2$ is bounded
  below by:
  \begin{equation}
    \label{eq:proof-linear}
    \vol_{T_1}(H)\bigl( n\ell_{T_2}(c_1) - (4B+4) \bigr) - 
    M(\vol_{T_2}(H) + 1)
  \end{equation}
  As $H$ is malnormal or cyclic, so is $\delta^n_1(H)$; hence at most
  one essential chain in $\CG^{\delta^n_1(H)}_2$ can be non-simply
  connected (Lemma \ref{lm:malnormal-freevolume}).  Thus
  $\vol_{T_2}(\delta_1^n(H))$ is bounded below by one less than
  \eqref{eq:proof-linear}. Thus for $C = 4B + M + 5$ the inequality
  \eqref{eq:lowerbound} holds.

  We will now sketch the proof for the upper bound
  \eqref{eq:upperbound}.  The idea is similar to the proof of
  \eqref{eq:lowerbound}.  Using the machinery developed above, we get
  an upper bound on the number of essential vertices and chains in
  $\CG_2^{\delta_1^n(H)}$ by looking at essential vertices and chains
  in $(\Upsilon_n^H)_{\Lambda_2}$.  When inserting a copy of $a_n$ at
  an essential vertex, this may prevent some folding that might have occurred originally in $(\CG_1^H)_{\Lambda_2} \to \Lambda_2$,
  causing some essential vertices and chains in $\CG_2^H$ to break into
  several essential vertices and chains in $\CG_2^{\delta_1^n(H)}$.
  This contribution to $\vol_{T_2}(\delta_1^n(H)$ is controlled by
  $M\vol_{T_1}(H)$.  Similar considerations apply at a crossing vertex
  of an essential chain.  At first glance, it may appear that the
  upper bound in \eqref{eq:upperbound} is too low as there may be
  several crossing vertices on a given essential chain.  However, it
  is easy to see that the contribution for all but the rightmost
  crossing vertex is folded onto to an essential vertex or chain that
  is already counted.  Hence \eqref{eq:upperbound} holds.
\end{proof}

It is likely that Theorem \ref{th:B} holds for ``multi-twists,'' i.e.,
products of Dehn twists arising from a single graph of groups
decomposition of $F_k$ with cyclic edge stabilizers.  This is the case
for surfaces, see \cite{Ivanov}.

\begin{ex}\label{ex:lb}
  We give an example that shows that the constant $C$ in
  \eqref{eq:lowerbound} is necessary.  Let $T_1$ be the cyclic tree for
  the splitting $F_3 = \la a,c \ra*_{\la c \ra} \la b,c \ra$ and $T_2
  = T_1\phi$ where $\phi$ is the outer automorphism of $F_3$
  represented by $a \mapsto b \mapsto c \mapsto ab$; in particular
  $\ell_{T_2}(c) = 2$.  For $g = ac^{-2}bc$ we have $\ell_{T_1}(g) =
  2$ and $\phi(g) = a^{-1}b^{-1}a^{-1}cab$ and hence $\ell_{T_2}(g) =
  4$.  Therefore, if $n = 2$ and $C = 0$, the right hand side of
  \eqref{eq:lowerbound} is $4$.  However, $\delta^2(g) = abc^{-1}$ and
  $\phi(\delta^2(g)) = bcb^{-1}a^{-1}$ and hence
  $\ell_{T_2}(\delta^2(g)) = 2$.  For the two bases $\CT_1 = \{ a,b,c
  \}$ and $\CT_2 = \{ ab,b,c \}$ the bounded cancellation constant
  $BCC(\CT_1,\CT_2)$ is $1$, and hence, from the proof of Theorem B, we
  see that we can choose $C = 10$.  Upon substituting, the right hand
  side of \eqref{eq:lowerbound} becomes $4n - 24$.  As $\phi(\delta^n(g)) = b(ab)^{n-2}c(ab)^{-(n-1)}$ is reduced for $n
  \geq 2$, we see that $\ell_{T_2}(\delta^n(g)) = 4n - 6$ for $n \geq
  2$.
\end{ex}

\section{Free factor ping pong}\label{sc:pingpong}

In this section we prove Theorem \ref{th:A}, using a variation on the familiar ping pong argument due to Hamidi-Tehrani . As the proof is
short, we include it here.

\begin{lem}[\cite{Hamidi-Tehrani}, Lemma 2.4]\label{pingpong}
  Let $G$ be a group generated by $g_1$ and $g_2$. Suppose that $G$
  acts on a set $\CX$, and that there is a function $| \cdot |: \CX
  \to \BR_{\geq 0}$ with the following properties: There are mutually
  disjoint subsets $\CX_1$ and $\CX_2$ of $\CX$ such that if $n>0$,
  then $g_i^{\pm n}(\CX - \CX_i) \subset \CX_i$, and for any $x \in
  \CX - \CX_i$ we have $|g_i^{\pm n}(x) | > |x|$. Then $G \cong F_2$,
  and the action on $\CX$ of every element $g \in G$ which is not
  conjugate to a power of some $g_i$ has no periodic points.
\end{lem}
\begin{proof}
  A non-empty reduced word in $g_1$ and $g_2$ is conjugate to a
  reduced word $w = g_1^{\epsilon_1} \cdots g_1^{\epsilon_2}$, where
  $\epsilon_1$ and $\epsilon_2$ are non-zero integers. If $x \in \CX -
  \CX_1$, then $w(x) \in \CX_1$; therefore $w(x) \neq x$ and $w$ is
  not the identity. If an element of $G$ which is not conjugate to a
  power of $g_1$ or $g_2$ has a periodic point, then some power of it
  has a fixed point. This power is conjugate to a reduced word of the
  form $w = g_i^{\epsilon_i} \cdots g_j^{\epsilon_j}$, with $i \neq j$
  and $\epsilon_i, \epsilon_j$ non-zero integers.  If $x \in \CX -
  \CX_j$, then by assumption $| w(x) | > |x|$. On the other hand, if
  $x \in \CX_j$, then $w^{-1}(x) = g_j^{- \epsilon_j} \cdots g_i^{-
    \epsilon_i}(x)$ so that $| w^{-1}(x) | > |x|$.  Hence $w$ does not
  have any fixed points and therefore no element of $G$ not conjugate
  to a power of $g_1$ or $g_2$ has a periodic point.
\end{proof}

Let $T_1$ and $T_2$ be filling very small cyclic trees with edge
stabilizers generated by conjugacy classes of the elements $c_1$ and
$c_2$, respectively. Also let $\delta_1$ and $\delta_2$ be the
associated Dehn twists, $M = M(k-1)$ from Proposition \ref{prop:safe}
and $C$ the larger of the constants $C(T_1,T_2)$ and $C(T_2,T_1)$ from
Theorem \ref{th:B}.  We let $\CX$ be the set of conjugacy classes of
proper free factors and cyclic subgroups of $F_k$.  Since the trees
$T_1$ and $T_2$ fill we have:
\[ \vol_{T_1}(X) + \vol_{T_2}(X) > 0 \] for any $X \in \CX$.  Choose
an irrational number $\lambda$ ($\lambda$ will be end up being close
to 1) and define sets:
\begin{align*}
  \CX_1 &= \{ X \in \CX \ | \ \vol_{T_1}(X) < \lambda \vol_{T_2}(X) \}
  \text{ and} \\
  \CX_2 &= \{ X \in \CX \ | \ \vol_{T_2}(X) < \lambda^{-1}
  \vol_{T_1}(X) \}.
\end{align*}
Hence $\CX$ is the disjoint union of $\CX_1$ and $\CX_2$.  Finally,
we define a function $| \cdot |\co \CX \to \BR_{\geq 0}$ by:
$$|X| = \vol_{T_1}(X) + \vol_{T_2}(X)$$

We will now show that for some $N$ and $m,n \geq N$, the group $\la
\delta_1^m,\delta_2^n \ra$ satisfies Lemma \ref{pingpong} with the set
$\CX$ and function $| \cdot |\co \CX \to \BR_{\geq 0}$.  The proof is
the same as for Lemma 3.1 in \cite{Hamidi-Tehrani}.

\begin{lem}\label{lm:pingpong}
  With the above notation:
 \begin{enumerate}
 \item $\delta_1^{\pm n}(\CX_2) \subset \CX_1$ if
   $n\ell_{T_2}(c_1) - C \geq (M+1)\lambda^{-1}$.
 \item If $n\ell_{T_2}(c_1) - C \geq (M+1)\lambda^{-1}$ and
   $X \in \CX_2$, then $|\delta_1^{\pm n}(X)| > |X|$.
 \item $\delta_2^{\pm n}(\CX_1) \subset \CX_2$ if
   $n\ell_{T_1}(c_2) - C \geq (M+1)\lambda$.
 \item If $n\ell_{T_1}(c_2) - C \geq (M+1)\lambda$ and $X
   \in \CX_1$, then $|\delta_2^{\pm n}(X)| > |X|$.
 \end{enumerate}
\end{lem}

\begin{proof}
  If $X \in \CX_2$, we have $\vol_{T_2}(X) <
  \lambda^{-1}\vol_{T_1}(X)$, and $\rank(X) \leq k-1$ and so by
  Theorem \ref{th:B}
\begin{align*} 
  \vol_{T_2}(\delta_1^{\pm n}(X)) &\geq \vol_{T_1}(X) \bigl(
  n\ell_{T_2}(c_1) - C\bigr) - M\vol_{T_2}(X) \\
  &> \vol_{T_1}(X) \bigl( n\ell_{T_2}(c_1) - C\bigr) -
  M\lambda^{-1} \vol_{T_1}(X) \\
  &= \vol_{T_1}(X) \bigl( n \ell_{T_2}(c_1)
  - C - M\lambda^{-1}\bigr) \\
  &= \vol_{T_1}(\delta_1^{\pm n}(X)) \bigl( n \ell_{T_2}(c_1)
  - C - M\lambda^{-1}\bigr) \\
  & \geq \lambda^{-1}\vol_{T_1}(\delta_1^{\pm n}(X))
\end{align*}
if $n\ell_{T_2}(c_1) - C \geq (M+1)\lambda^{-1}$, and hence
$\delta_1^{\pm n}(X) \in \CX_1$.  This shows $(1)$, and statement
$(3)$ is similar.  If $X \in \CX_2$, we have $\vol_{T_2}(X) <
\lambda^{-1}\vol_{T_1}(X)$, and $\rank(X) \leq k-1$, so again by
Theorem \ref{th:B}:
\begin{align*}
  |\delta_1^{\pm n}(X)| & = \vol_{T_1}(\delta_1^{\pm n}(X)) +
  \vol_{T_2}(\delta_1^{\pm n}(X)) \\
  & \geq \vol_{T_1}(X) + \vol_{T_1}(X)\bigl(n \ell_{T_2}(c_1)
  -  C \bigr) - M\vol_{T_2}(X) \\
  & > \vol_{T_1}(X)\bigl(n\ell_{T_2}(c_1) -
  C + 1\bigr) - M\lambda^{-1}\vol_{T_1}(X) \\
  &= \vol_{T_1}(X)\bigl(n\ell_{T_2}(c_1) - C + 1
  - M\lambda^{-1} \bigr) \\
  &\geq \vol_{T_1}(X)\bigl( 1 + \lambda^{-1}\bigr) \\
  & = \vol_{T_1}(X) + \lambda^{-1}\vol_{T_1}(X) \\
  & > \vol_{T_1}(X) + \vol_{T_2}(X) = |X|
\end{align*}
if $n\ell_{T_2}(c_1) - C \geq (M+1)\lambda^{-1}$.  This shows $(2)$,
and statement $(4)$ is similar.
\end{proof}

\noindent Equipped with this lemma, we are now ready to prove our main
result.

\begin{thm}\label{th:A}
  Let $\delta_1$ and $\delta_2$ be the Dehn twists of
  $F_k$ for two filling cyclic splittings of $F_k$. Then
  there exists $N = N(\delta_1,\delta_2)$ such that for $m,n >N$:
  \begin{enumerate} 
  \item $\la \delta_1^m, \delta_2^n \ra$ is isomorphic to the
    free group on two generators; and
  
  \item if $\phi \in \la \delta_1^m, \delta_2^n \ra$ is not conjugate
    to a power of either $\delta^m_1$ or $\delta^n_2$, then $\phi$ is
    a hyperbolic fully irreducible element of $\Out F_k$.
  \end{enumerate}
\end{thm}

\begin{proof}
  As mentioned in Remark \ref{remark:verysmall}, without loss of
  generality, we can assume that $\delta_1$ and $\delta_2$ are
  associated to very small cyclic trees for $F_k$.  Using the above
  set-up and notation, let $\lambda$ be an irrational number such that
  $\max\{ \lambda,\lambda^{-1}\} \leq 2$.  Because $\lambda$ is
  irrational, the set $\CX$ is equal to the disjoint union $\CX_1
  \sqcup \CX_2$.  Let $N$ by large enough such that:
  \begin{equation*}
    N\ell_{T_2}(c_1) - C \geq 2(M+1) \mbox{ and } 
    N\ell_{T_1}(c_2) - C \geq 2(M+1).
  \end{equation*}
  Then Lemma \ref{lm:pingpong} implies that for $m,n \geq N$, the
  action of the group $\la \delta_1^m,\delta_2^n\ra$ on $\CX$
  satisfies the hypotheses of Lemma \ref{pingpong} with the function
  $|X| = \vol_{T_1}(X) + \vol_{T_2}(X)$.  It follows that $\la
  \delta_1^m,\delta_2^n \ra \simeq F_2$. Further, the Lemma
  \ref{pingpong} implies that if $\phi \in \la \delta_1^m, \delta_2^n
  \ra$ is not conjugate to a power of either $\delta^m_1$ or
  $\delta^n_2$ then $\phi$ acts on $\CX$ without periodic orbits.  As
  $\CX$ contains all of the conjugacy classes of proper free factors,
  $\phi$ is fully irreducible. Likewise, as $\CX$ contains all of the conjugacy
  classes of cyclic subgroups, $\phi$ is hyperbolic.
\end{proof}
\begin{rmk}\label{rm:directpingpong}
  By applying the ping pong argument using Lemma \ref{lm:pingpong}
  directly to the word $w =\delta_1^{\epsilon_1}\delta_2^{\kappa_1}
  \cdots \delta_1^{\epsilon_n}\delta_2^{\kappa_n}$ where $n \geq 2$,
  and $|\epsilon_i|,|\kappa_i| \geq N$, except possibly $\epsilon_1 =
  0$ or $\kappa_n = 0$, we can see that $w$ is nontrivial.
  Additionally, if both $|\epsilon_1|$ and $|\kappa_n|$ are equal
  to 0, or if both are at least $N$, then $w$ is a fully irreducible
  hyperbolic element of $\Out F_k$.
\end{rmk}

\begin{rmk}\label{rm:Mangahas}
  Inspired by Hamidi-Tehrani's approach, Mangahas \cite{Mangahas}
  proved that subgroups of the mapping class group have uniform
  exponential growth with a uniform bound depending only on the
  surface and not on the subgroup. It is possible that Theorem
  \ref{th:A} is a step towards proving Mangahas' theorem for $\Out
  F_k$, although much of the machinery she uses for the mapping class
  group is still undeveloped in the $\Out F_k$ setting.
\end{rmk}

\section{Coarse biLipschitz equivalence}\label{sc:bilip}

Using the techniques developed in Sections \ref{counting} and
\ref{sc:growth} we can now prove that the sum of the free volumes of a
proper free factor for two filling very small cyclic trees is
biLipschitz equivalent to the free volume of the free factor for any
tree in Outer space.  Kapovich and Lustig showed this equivalence for
a cyclic subgroup \cite{KL_Intersection}.

\begin{thm}\label{th:C}
  Let $T_1$ and $T_2$ be two very small cyclic trees for $F_k$ that
  fill, and let $T \in cv_k$. Then there is a constant $K$ such that for
  any proper free factor or cyclic subgroup $X \subset F_k$:
  \begin{equation}\label{eq:bilip}
    \frac{1}{K}\vol_T(X) \leq \vol_{T_1}(X) + \vol_{T_2}(X) \leq K\vol_T(X).
  \end{equation}
\end{thm}

\begin{proof}
  First, recall that for any trees $T$ and $T'$ in $cv_k$, there is a
  constant $K_0$ such that for any free factor or cyclic group $X$
  \[ \frac{1}{K_0} \vol_T(X) \leq \vol_{T'}(X) \leq K_0 \vol_T(X). \]
  Thus to prove \eqref{eq:bilip} we can just let $T$ be the tree
  $\tilde{\Lambda}_1$, where $\Lambda_1 = \Lambda_{\CT_1}$ and $\CT_1$
  is a basis for $F_k$ relative to $T_1$, metrized such that every
  edge has length 1.  Further consider the tree $\tilde{\Lambda}_2$, where $\Lambda_2 = \Lambda_{\CT_2}$, and where $\CT_2$ is a basis for $F_k$
  relative to $T_2$, again metrized such that every edge has length 1.

  Fix a constant $K_1$ such that for any free factor or cyclic
  subgroup $X$
  \[ \frac{1}{K_1} \vol_{\tilde{\Lambda}_1}(X) \leq
  \vol_{\tilde{\Lambda}_2}(X) \leq K_1 \vol_{\tilde{\Lambda}_1}(X).\]
  As the rank of $X$ is bounded, each edge in $\tilde{\Lambda}_1^X/X$
  and $\tilde{\Lambda}_2^X/X$ can be contained in only boundedly many
  essential chains independent of $X$. By Theorems
  \ref{prop:amalgam-free-volume} and \ref{prop:hnn-free-volume}, for
  $i=1,2$, the free volume $\vol_{T_i}(X)$ is equal to the total
  number of simply connected chains and essential vertices of
  $\tilde{\Lambda}_i^X/X$.  Thus this is less than some constant $D$
  times the total number of edges, $\vol_{\tilde{\Lambda}_i}(X)$, of
  the graph $\tilde{\Lambda}_i^X/X$. Thus we have $\vol_{T_1}(X) \leq
  D\vol_{\tilde{\Lambda}_1}(X)$ and $\vol_{T_2}(X) \leq
  D\vol_{\tilde{\Lambda}_2}(X)$, and hence
  \begin{align*}
    \vol_{T_1}(X) + \vol_{T_2}(X) & \leq D\vol_{\tilde{\Lambda}_1}(X) +
    D\vol_{\tilde{\Lambda}_2}(X) \\
    & \leq D\vol_{\tilde{\Lambda}_1}(X) + DK_1 \vol_{\tilde{\Lambda}_1}(X) \\
    & = D(K_1 + 1) \vol_{\tilde{\Lambda}_1}(X)
  \end{align*}
  which shows the right hand inequality of \eqref{eq:bilip}.

  By \cite[Theorem 1.4]{KL_Intersection}, there exists a constant $K'$
  such that for $g \in F_k$:
  \[ \frac{1}{K'} \ell_{\tilde{\Lambda}_1}(g) \leq \ell_{T_1}(g) +
  \ell_{T_2}(g) \leq K' \ell_{\tilde{\Lambda}_1}(g). \] This is
  \eqref{eq:bilip} when $X = \la g \ra$.
  
  Otherwise, as $X$ is a proper (noncyclic) free factor, deleting
  vertices of $\Lambda_1^X = \tilde{\Lambda}_1^X/X$ with valence $\geq
  3$ results in at most $3k - 3$ segments; denote these segments by
  $S(\Lambda^X_1)$.  Then for each such segment $\alpha \in
  S(\Lambda^X_1)$, there is a subsegment $\alpha' \subseteq \alpha$
  such that $|\alpha'|_{\CT_1} \geq \frac{1}{2}|\alpha|_{\CT_1}$ and
  $\alpha'$ is cyclically reduced with respect to $\CT_1$.  Hence:
  \[ \vol_{\tilde{\Lambda}_1}(X) = \sum_{\alpha \in S(\Lambda_1^X)}
  |\alpha|_{\CT_1} \leq 2\sum_{\alpha \in S(\Lambda_1^X)}
  |\alpha'|_{\CT_1} = 2\sum_{\alpha \in S(\Lambda_1^X)}
  \ell_{\tilde{\Lambda}_1}(\alpha').\] For each such $\alpha'$, let
  $\alpha'_{\Lambda_2}$ be its image under graph composition using the
  change of marking homotopy equivalence $\nu \co \Lambda_1 \to
  \Lambda_2$.  We can get a lower bound on $\vol_{T_1}(X) +
  \vol_{T_2}(X)$ by estimating the sum of the number of essential pieces in
  the segments $\alpha'$ and the number of essential pieces in the
  segments $\alpha'_{\Lambda_2}$ that survive after folding
  $(\Lambda_1^X)_{\Lambda_2} \to \Lambda_2$.  Notice that:
  \begin{align*}
    &\sum_{\alpha \in S(\Lambda^X_1)} \#\mbox{essential pieces in }
    \alpha' + \#\mbox{essential pieces in } \alpha'_{\Lambda_2} \\
    & \hskip 5 cm =\sum_{\alpha \in S(\Lambda^X_1)} \ell_{T_1}(
    \alpha') + \ell_{T_2}(\alpha') \\
    & \hskip 5 cm \geq \frac{1}{K'} \sum_{\alpha \in S(\Lambda^X_1)}
    \ell_{\tilde{\Lambda}_1}(\alpha') \\
    & \hskip 5 cm \geq \frac{1}{2K'}\vol_{\tilde{\Lambda}_1}(X)
  \end{align*}

  Let $B = BCC(\CT_1,\CT_2)$.  As in Section \ref{sc:growth}, we 
  lose at most the extremal $B$ edges of $\alpha'_{\Lambda_2}$ whilst
  folding and pruning $(\Lambda_1^X)_{\Lambda_2} \to \Lambda_2$,
  thus eliminating at most $2B+2$ essential pieces from
  $\alpha_{\Lambda_2}$.  Thus we have
  \begin{equation*}
    \vol_{T_1}(X) + \vol_{T_2}(X) \geq \frac{1}{2K'} 
    \vol_{\tilde{\Lambda}_1}(X) - (2B+2)(3k - 3).
  \end{equation*}
  In other words:
  \begin{equation*}
    \vol_{T_1}(X) + \vol_{T_2}(X)\bigl(1 + (2B+2)(3k - 3) \bigr) 
    \geq \frac{1}{2K'} \vol_{\tilde{\Lambda}_1}(X)
  \end{equation*}
  as $\vol_{T_1}(X) + \vol_{T_2}(X) \geq 1$.  Choosing $K = \max \{
  K_1 + 1, 2K'\bigl(1 + (2B+2)(3k-3)\bigr) \}$ completes the proof.
\end{proof}

\bibliography{bibliography}

\bibliographystyle{siam}

\end{document}